\documentclass[12pt]{amsart}  

\usepackage{graphicx}
\usepackage{here} 
\usepackage{array} 
\usepackage[b]{esvect}
\usepackage{dsfont}
\usepackage{bm}
\usepackage{bbm}

\usepackage{vmargin}
\usepackage{amssymb}
\usepackage{mathrsfs}
\usepackage[all]{xy}
\usepackage[usenames,dvipsnames]{xcolor}
\usepackage{soul}
\RequirePackage[colorlinks,linkcolor=blue,citecolor=LimeGreen,urlcolor=red]{hyperref} 
\usepackage{amsmath}

\newtheorem{theorem}{Theorem}[section]

\newtheorem{lemma}[theorem]{Lemma}

\newtheorem{prop}[theorem]{Proposition}
\newtheorem{remark}[theorem]{Remark}

\numberwithin{equation}{section}

\newcommand{\R}{\mathbb{R}}

\newcommand{\N}{\mathbb{N}}

\newcommand{\T}{\mathbb{T}}

\newcommand{\func}[3]{#1 : #2 \longrightarrow #3}

\newcommand{\disp}{\displaystyle}
\newcommand{\abs}[1]{\left|#1\right|}
\newcommand{\eps}{\varepsilon}
\newcommand{\norm}[1]{\left\|#1\right\|}

\renewcommand{\leq}{\leqslant}
\renewcommand{\geq}{\geqslant}
\renewcommand{\bar}{\overline}
\renewcommand{\tilde}{\widetilde}
\newcommand{\pa}[1]{\left(#1\right)}
\newcommand{\pab}[1]{\big(#1\big)}
\newcommand{\pabb}[1]{\Big(#1\Big)}
\newcommand{\cro}[1]{\left[#1\right]}

\newcommand{\scalprod}[1]{\left\langle#1\right\rangle}

\newcommand\restr[2]{{
  \left.\kern-\nulldelimiterspace 
  #1 
  \right|_{ #2} 
  }}

\newcommand{\dd}{\mathrm{d}}            

\newcommand{\init}{{\hspace{0.5mm}\mathrm{in}}}
\newcommand{\inits}{{\hspace{0.1mm}\mathrm{in}}}
\newcommand{\ms}{\text{{\tiny MS}}}

\newcommand{\Ker}{\mathrm{ker\hspace{0.5mm}}}
\newcommand{\Span}{\mathrm{Span}}

\newcommand{\VV}{\mathbf{V}_{\tilde{U}}}

\newcommand{\F}{\mathcal{F}}

\newcommand{\boldc}{\mathbf{c}}
\newcommand{\boldu}{\mathbf{u}}

\newcommand{\sobolevx}[1]{H_x^{#1}\big(\bar{\boldc}^{-\frac{1}{2}}\big)}

\def\restriction#1#2{\mathchoice
              {\setbox1\hbox{${\displaystyle #1}_{\scriptstyle #2}$}
              \restrictionaux{#1}{#2}}
              {\setbox1\hbox{${\textstyle #1}_{\scriptstyle #2}$}
              \restrictionaux{#1}{#2}}
              {\setbox1\hbox{${\scriptstyle #1}_{\scriptscriptstyle #2}$}
              \restrictionaux{#1}{#2}}
              {\setbox1\hbox{${\scriptscriptstyle #1}_{\scriptscriptstyle #2}$}
              \restrictionaux{#1}{#2}}}
\def\restrictionaux#1#2{{#1\,\smash{\vrule height .8\ht1 depth .85\dp1}}_{\,#2}}

\makeatletter
\def\namedlabel#1#2{\begingroup
    #2%
    \def\@currentlabel{#2}%
    \phantomsection\label{#1}\endgroup
}
\makeatother

\def\signmarc{\bigskip \begin{center} {\sc
Marc Briant\par\vspace{3mm}
Universit\'e de Paris \par 
Laboratoire MAP5, UMR CNRS 8145 \par 
F-75006 Paris, France \par
\vspace{3mm}
e-mail:} \tt{briant.maths@gmail.com} \end{center}}

\def\signandrea{\bigskip \begin{center} { \sc
Andrea Bondesan \par\vspace{3mm}
University of Graz \par \vspace{1mm}
Institute of Mathematics and Scientific Computing \par \vspace{1mm}
8010 Graz, Austria \par
\vspace{3mm}
{\sc e-mail:}} \tt{andrea.bondesan@gmail.com} \end{center}}

\begin{document} 

\title[Perturbative Incompressible Maxwell-Stefan]{Perturbative Cauchy theory for a\\flux-incompressible Maxwell-Stefan system}
\author{Andrea Bondesan and Marc Briant}
\thanks{The authors would like to thank Laurent Boudin and B\'er\'enice Grec for fruitful discussions on the theory of gaseous and fluid mixtures.}

\begin{abstract}
Recently, the authors proved \cite{BonBri} that the Maxwell-Stefan system with an incompressibility-like condition on the total flux can be rigorously derived from the multi-species Boltzmann equation. Similar cross-diffusion models have been widely investigated, but the particular case of a perturbative incompressible setting around a non constant equilibrium state of the mixture (needed in \cite{BonBri}) seems absent of the literature. We thus establish a quantitative perturbative Cauchy theory in Sobolev spaces for it. More precisely, by reducing the analysis of the Maxwell-Stefan system to the study of a quasilinear parabolic equation on the sole concentrations and with the use of a suitable anisotropic norm, we prove global existence and uniqueness of strong solutions and their exponential trend to equilibrium in a perturbative regime around any macroscopic equilibrium state of the mixture. As a by-product, we show that the equimolar diffusion condition naturally appears from this perturbative incompressible setting.
\end{abstract}

\maketitle

\vspace*{10mm}



\tableofcontents

\section{Introduction}

\noindent We consider a chemically non-reacting ideal gaseous mixture composed of $N \geq 2$ different species, having atomic masses $(m_i)_{1\leq i\leq N}$ and evolving in the $3$-dimensional torus $\T^3$. We assume isothermal and isobaric conditions, focusing our attention on a purely diffusive setting. For any $1\leq i\leq N$, the balance of mass links the molar concentration $c_i=c_i(t,x)$ of the $i$-th species to its molar flux $\F_i=\F_i(t,x)$ \textit{via} the continuity equation
\begin{equation}\label{Continuity equation}
\partial_t c_i + \nabla_x\cdot \F_i = 0 \quad \textrm{on }  \R_+\times\T^3.
\end{equation}

Let $c=\sum_i c_i$ denote the total molar concentration of the mixture and set $n_i = c_i/c$, the mole fraction of the $i$-th species. The Maxwell-Stefan equations give relations between the molar fluxes and the mole fractions and read, for any~${1\leq i\leq N}$,
\begin{equation}\label{MS equation}
- c \nabla_x n_i = \sum\limits_{j=1}^N \frac{n_j \F_i - n_i \F_j}{D_{ij}} \quad \textrm{on } \R_+\times\T^3,
\end{equation}
where $D_{ij}=D_{ji}>0$ are the diffusion coefficients between species $i$ and~$j$.

Independently introduced in the 19th century by Maxwell \cite{Max} for dilute gases and Stefan \cite{Ste} for fluids, equations $\eqref{MS equation}$ describe the cross-diffusive interactions inside a mixture and therefore lie in the class of the so-called cross-diffusion models \cite{ShiKawTer, LouNi, LouMar, DesLepMou, Jun}. In particular, the system $\eqref{MS equation}$ gives a generalization \cite{KriWes} of Fick's law of mono-species diffusion \cite{Fic}, making it of core importance for applications in physics and medicine, where it can be used for example to model the propagation of polluting particles in the air or to characterize the gas exchanges in the lower generations of the human lung \cite{TDBCH,Cha,BouGotGre}. Besides, the Maxwell-Stefan equations also raise a great theoretical interest, as their mathematical analysis appears to be very challenging. The difficulties come from the fact that summing over $i$ the relations $\eqref{MS equation}$, we obtain a linear dependence on the mole fractions' gradients which imposes to introduce a further condition in order to close the system and provide a satisfactory Cauchy theory to \eqref{Continuity equation}--\eqref{MS equation}. To our knowledge, the existing mathematical results that deal with the problem of existence and uniqueness of solutions to the sole system \eqref{Continuity equation}--\eqref{MS equation} are all tied up to the assumption that the mixture is subject to a transient equimolar diffusion \cite{KriWes}, namely the total diffusive flux satisfies
\begin{equation}\label{MS closure}
\sum_{i=1}^N \F_i(t,x) = 0, \quad t\geq 0,\ x\in\T^3.
\end{equation} 
Concerning the local Cauchy problem for such multicomponent systems, a first general result \cite{GioMas} was obtained by Giovangigli and Massot for compressible reactive fluids (including viscous and energy equations). They proved local existence and uniqueness of smooth solutions in the whole space, starting from general initial data. Later on, working in a bounded domain~$\Omega$, Bothe exploited classical results \cite{Ama} from the theory of quasilinear parabolic equations in order to show \cite{Bot} local-in-time existence and uniqueness for solutions in~$L^p(\Omega)$, also starting from a general initial datum. Similar outcomes have been obtained through the use of these techniques by Herberg \textit{et al.} in a mass-based (rather than molar-based) setting with chemical reactions \cite{HMPW17} and by Hutridurga and Salvarani in a non-isothermal setting \cite{HutSal2}. The first global existence result was established by Giovangigli in \cite{Gio1}, looking at a perturbative regime where the initial datum is sufficiently close to a constant stationary state of the mixture. He proved global existence, uniqueness and trend to equilibrium in Sobolev spaces on $\R^3$. Boudin \textit{et al.} investigated in \cite{BouGreSal1} the particular case of a 3-species mixture, when two diffusion coefficients are equal: the authors were able to establish global existence and uniqueness of smooth solutions for~$L^\infty(\Omega)$ initial data, as well as their long-time convergence towards the corresponding constant equilibrium state. By passing to entropic variables, J\"ungel and Stelzer obtained \cite{JunSte} global-in-time existence of weak solutions in $H^1(\Omega)$ as well as their exponential decay to the homogeneous steady state of the mixture, for arbitrary diffusion coefficients and for general initial data. At last, we mention that the global existence of weak solutions in $H^1(\Omega)$ has also been shown to hold in more intricate problems where the Maxwell-Stefan system is coupled \cite{CheJun, MarTem} with the incompressible Navier-Stokes equation (which is used to describe the evolution of the mass average velocity of the mixture) or when chemical reactions are taken into account~\cite{MarTem, DauJunTan}. In particular, the entropy method exploited in \cite{CheJun, DauJunTan} also allowed to recover the exponential decay of solutions towards equilibrium.

The present article aims at studying the problem of existence and uniqueness of perturbative solutions to an incompressible variant of the Maxwell-Stefan system \eqref{Continuity equation}--\eqref{MS equation}--\eqref{MS closure}, which is written for any $1\leq i\leq N$ on $\R_+\times\T^3$ in terms of the species' mean velocities $(u_i)_{1\leq i\leq N}$ as
\begin{gather}
\partial_t c_i + \nabla_x \cdot \pa{c_i u_i} = 0, \label{MS mass}
\\[1mm] - \nabla_x c_i = \sum\limits_{j=1}^N c_ic_j\frac{u_i - u_j}{\Delta_{ij}},  \label{MS momentum}
\\[1mm] \nabla_x\cdot \pa{\sum\limits_{i=1}^N c_i u_i} =0, \label{MS incompressibility}
\end{gather}
where the closure relation $\eqref{MS closure}$ is replaced by the incompressibility-like condition $\eqref{MS incompressibility}$. Note in particular that the model \eqref{Continuity equation}--\eqref{MS equation} can be easily recovered \textit{a priori} from \eqref{MS mass}--\eqref{MS momentum}--\eqref{MS incompressibility} by defining $\F_i = c_i u_i$ for any $1\leq i\leq N$ and by supposing that $\sum_i c_i^\init(x)=const.$ on $\T^3$. Indeed, thanks to this hypothesis, the total number number of particles $c$ remains constant over time on $\T^3$, since both $\partial_t c =0$ and $\nabla_x c=0$ are obtained from \eqref{MS mass}--\eqref{MS momentum}--\eqref{MS incompressibility}. The quantities $\Delta_{ij}$ are then linked to the diffusion coefficients through the relations $D_{ij} = \Delta_{ij}/c$.

The above Maxwell-Stefan-type system is of peculiar significance, as recent works \cite{BouGreSal2, HutSal1, BouGrePav} managed to formally derive it starting from the kinetic equations. In particular it is worth mentioning that, even if the equimolar condition $\eqref{MS closure}$ is systematically assumed to hold as being a specific feature which is intrinsic to the physics of diffusion \cite{ChaCow, KriWes}, up to now no asymptotics has been able to recover it from the kinetic level, leaving open the question of its mathematical relevance. In fact, while $\eqref{MS closure}$ obviously implies the incompressibility condition $\eqref{MS incompressibility}$, the contrary is true only in a one space dimension setting. 

In a companion paper \cite{BonBri}, starting from the Boltzmann multi-species equation and supposing to have the proper bounds and regularity on the solutions of \eqref{MS mass}--\eqref{MS momentum}--\eqref{MS incompressibility}, we were able to make the formal asymptotics \cite{BouGreSal2, HutSal1, BouGrePav} rigorous. Providing a Cauchy theory for \eqref{MS mass}--\eqref{MS momentum}--\eqref{MS incompressibility} thus becomes crucial if one wants to deal with such rigorous hydrodynamical derivation and, by this, show the mathematical coherence between the mesoscopic and the macroscopic descriptions.

As usually done in the literature \cite{Gio1,Bot,JunSte}, we begin by introducing the matrix
\begin{equation}\label{MS matrix}
A\pa{\mathbf{c}} = \pa{\frac{c_i c_j}{\Delta_{ij}} - \pa{\sum\limits_{k=1}^{N}\frac{c_i c_k}{\Delta_{ik}}}\delta_{ij}}_{1\leq i,j \leq N},
\end{equation}
which depends in a nonlinear way on the concentrations  $(c_i)_{1\leq i\leq N}$. In this way, the system of equations \eqref{MS mass}--\eqref{MS momentum}--\eqref{MS incompressibility} can be initially rewritten in a more convenient vectorial form as
\begin{gather*}
\partial_t\mathbf{c}+ \nabla_x\cdot\pa{\mathbf{c}\mathbf{u}} = 0,
\\[3mm]     \nabla_x \mathbf{c} = A(\mathbf{c})\mathbf{u},
\\[3mm]    \nabla_x \cdot \langle \mathbf{c}, \mathbf{u} \rangle = 0,
\end{gather*}
where bold letters will denote $N$-vectors referring to the species of the mixture, so that in this case $\mathbf{c} = (c_1,\dots,c_N)$ and $\mathbf{u} =(u_1,\dots,u_N)$, the product $\mathbf{c}\mathbf{u}$ has to be understood componentwise and the notation $\langle \cdot, \cdot \rangle$ indicate the standard Euclidean scalar product in $\R^N$. A natural idea for tackling the problem would then be to invert the gradient relation in order to express $\mathbf{u}$ in terms of $\mathbf{c}$ and obtain an evolution equation for the sole unknown $\mathbf{c}$, by replacing $\mathbf{u} = A(\mathbf{c})^{-1}\nabla_x\mathbf{c}$ into the continuity equation. Unfortunately, it is possible to prove \cite{Gio1, Bot, JunSte} that the matrix $A$ is only negative semi-definite, with $\Ker A = \Span (\mathbf{1})$. 
Therefore, any existing Cauchy theory for the Maxwell-Stefan equations is based on the possibility of explicitly computing the pseudoinverse of $A$, which is defined on the space $\pab{\Span(\mathbf{1})}^\perp$. This can be achieved \cite{Gio2, Gio1, Bot, JunSte, CheJun, HutSal2, DauJunTan} using the Perron-Frobenius theory for quasi-positive matrices. However, a drawback of this strategy is that the computations giving the explicit form of $A^{-1}$ are extremely intricate and do not offer a neat understanding of the action of $A$ on the velocities $\mathbf{u}$. As already pointed out, since one cannot see the part of $\mathbf{u}$ that evolves in $\Ker A$, a closure assumption of type $\eqref{MS closure}$ is needed in order to compensate this lack of informations.

In this work we propose another approach which takes inspiration from the micro-macro decomposition techniques commonly used in the kinetic theory of gases. More precisely, by defining the orthogonal projection $\pi_A$ onto $\Span(\mathbf{1})$, associated to the non-injective operator $A$, we split $\mathbf{u}=\pi_A(\mathbf{u})+\mathbf{U}$ into a part projected onto $\Span(\mathbf{1})$ and an orthogonal part $\mathbf{U}$ which is projected onto $\pab{\Span(\mathbf{1})}^\perp$. Using the incompressibility condition $\eqref{MS incompressibility}$ we construct a new system of equations, equivalent to \eqref{MS mass}--\eqref{MS momentum}--\eqref{MS incompressibility} for full velocities $\mathbf{u}$, in which the Maxwell-Stefan matrix only acts on $\mathbf{U}$
\begin{gather}
\partial_t \mathbf{c} +  \nabla_x\cdot\pa{\mathbf{c}\mathbf{V}_U} +  \bar{u} \cdot \nabla_x \mathbf{c} = 0, \label{MS mass new}
\\[3mm]       \nabla_x\mathbf{c} = A(\mathbf{c})\mathbf{U}, \label{MS momentum new}
\end{gather}
where $\bar{u}: \R_+\times\T^3 \to \R^3$ is a divergence-free function inherited from $\eqref{MS incompressibility}$, the vector~$\mathbf{V}_U$ is linked to the orthogonal component $\mathbf{U}$ \textit{via} the relation $\mathbf{V}_U = \mathbf{U} - \frac{\langle \mathbf{\mathbf{c}}, \mathbf{U}\rangle}{\langle \mathbf{\mathbf{c}},\mathbf{1}\rangle}\mathbf{1}$ and the full velocity $\mathbf{u}$ is finally reconstructed as~$\mathbf{u} = \mathbf{\bar{u}} +  \mathbf{V}_U$. Note that the above reformulation is very similar to the system investigated by Chen and J\"ungel, where the role played here by $\bar{u}$ is the same as the one played by the mass average velocity, solution to the incompressible Navier-Stokes equation in \cite{CheJun}. Indeed, quite surprisingly, it turns out that the kinetic decomposition naturally leads to the standard splitting between convection and diffusion velocities \cite{ChaCow, KriWes, CheJun}, as we shall see that $\bar{u} = \frac{\langle \mathbf{\mathbf{c}}, \mathbf{u}\rangle}{\langle \mathbf{\mathbf{c}},\mathbf{1}\rangle}$ actually coincides with the molar average velocity of the mixture while the vector $\mathbf{V}_U$ satisfies the relation $\langle \mathbf{c}, \mathbf{V}_U \rangle = 0$, equivalent to $\eqref{MS closure}$. In particular, we wish to point out that the equimolar diffusion condition we recover is an intrinsic property of the model, which arises mathematically.

Since $\mathbf{U} \in \pab{\Span(\mathbf{1})}^\perp$ in \eqref{MS mass new}--\eqref{MS momentum new}, the pseudoinverse of $A$ is now well-defined. However, as opposed to the entropy method of \cite{CheJun}, we make use here of a hypocoercive strategy which exploits the properties of $A$ without the need of computing its explicit structure. Moreover, instead of eliminating one last species \cite{JunSte, CheJun}, a symmetric role is given to all the species' variables, as in \cite{MarTem}. Our approach also provides an original point of view which exhibits a clear separation between $\pi_A(\mathbf{u})$ and $\mathbf{U}$, allowing to show explicitly the actual action of $A$ on the sole vector $\mathbf{U}$. We shall thus prove that the orthogonal reformulation \eqref{MS mass new}--\eqref{MS momentum new} in terms of the couple $(\mathbf{c},\mathbf{U})$ is fully closed and has a quasilinear parabolic structure. With the use of a suitable Sobolev anisotropic norm we shall subsequently establish a negative feedback coming from the Maxwell-Stefan operator $\eqref{MS matrix}$. This fact will allow to derive the \textit{a priori} energy estimates leading to global-in-time existence and uniqueness (in a perturbative sense) of strong solutions $(\mathbf{c},\mathbf{U})$ to  \eqref{MS mass new}--\eqref{MS momentum new} and, eventually, the same result will hold for the couple $(\mathbf{c},\mathbf{u})$ with full velocities, solution of the original system \eqref{MS mass}--\eqref{MS momentum}--\eqref{MS incompressibility}. 


\smallskip
In the next section we present all the notations and we state our main theorem. Section \ref{sec:MS matrix} is then dedicated to the investigation of the fundamental properties (spectral gap and some Sobolev estimates) of the Maxwell-Stefan matrix $A(\mathbf{c})$. At last, in Section \ref{sec:perturbed IMS} we shall prove our main result.

\section{Main result}\label{sec:Main results}

\subsection{Notations and conventions}
Let us first introduce the main notations that we use throughout the paper. Vectors and vector-valued operators in $\R^N$ will be denoted by a bold symbol, whereas their components will be denoted by the same indexed symbol. For instance, $\mathbf{w}$ represents the vector or vector-valued operator $(w_1,\dots,w_N)$. In particular, we shall use the symbol $\mathbf{1}$ to name the specific vector $(1,\dots,1)$. Henceforth, the multiplication of $N$-vectors has to be understood in a component by component way, so that for any $\mathbf{w}, \mathbf{W}\in\R^N$ and any $q\in\mathbb{Q}$ we have
\begin{equation*}
\mathbf{w}\mathbf{W} = (w_i W_i)_{1\leq i\leq N},\qquad \mathbf{w}^q=(w_i^q)_{1\leq i\leq N}.
\end{equation*}
Moreover, we introduce the Euclidean scalar product in $\R^N$ weighted by a vector $\mathbf{w}\in (\R^*_+)^N$, which is defined as
$$\langle \mathbf{c},\mathbf{d}\rangle_{\mathbf{w}}= \sum\limits_{i=1}^N c_i d_i w_i,$$
and induces the norm $\norm{\boldc}_{\mathbf{w}}^2=\scalprod{\boldc,\boldc}_{\mathbf{w}}$. When $\mathbf{w}=\mathbf{1}$, the index $\mathbf{1}$ will be dropped in both the notations for the scalar product and the norm.

The convention we choose for the functional spaces is to index the space by the name of the concerned variable. For $p$ in $[1,+\infty]$ we have
$$L^p_{t} = L^p (0,+\infty),\qquad L^p_x = L^p\left(\T^3\right), \qquad L^p_{t,x} = L^p\left(\R^+\times\T^3\right).$$
To any positive measurable function $\func{\mathbf{w}}{\T^3}{\pab{\R^*_+}^N}$ in the variable $x$, we associate the weighted Hilbert space $L^2(\T^3,\mathbf{w})$, which is defined by the scalar product and norm
\begin{gather*}
\scalprod{\boldc,\mathbf{d}}_{L^2_x(\mathbf{w})} = \sum_{i=1}^N\langle c_i,d_i\rangle_{L_x^2(w_i)}=\sum\limits_{i=1}^N \int_{\T^3} c_i d_i w_i^2 \dd x,
\\[3mm]     \|\boldc\|_{L_x^2(\mathbf{w})}^2=\sum_{i=1}^N\| c_i\|^2_{L_x^2(w_i)} = \sum\limits_{i=1}^N \int_{\T^3} c_i^2 w_i^2 \dd x.
\end{gather*}
Finally, in the same way we can introduce the corresponding weighted Sobolev spaces. Consider a multi-index $\alpha\in\N^3$ of length $|\alpha|=\sum_{k=1}^3 \alpha_k$. For any $s\in\N$ and any vector-valued function $\boldc\in H^s(\T^3,\mathbf{w})$, we define the norm
\begin{equation*}
\norm{\boldc}_{H^s_x(\mathbf{w})} = \left(\sum\limits_{i=1}^N \sum\limits_{|\alpha|\leq s}\norm{\partial_x^\alpha c_i}^2_{L^2_{x}(w_i)}\right)^{1/2}.
 \end{equation*}


\subsection{Main theorem}

We build up a Cauchy theory for the incompressible Maxwell-Stefan system \eqref{MS mass}--\eqref{MS momentum}--\eqref{MS incompressibility} perturbed around any macroscopic equilibrium state of the form $(\bar{\boldc},\bar{\boldu})$, where $\bar{\boldc}\in\pab{\R_+^*}^N$ is a positive constant $N$-vector and $\mathbf{\bar{u}} = (\bar{u},\dots,\bar{u})$ is such that the velocity vector $\bar{u}:\R_+\times\T^3\to\R^3$ is common to all the species and satisfies $\nabla_x\cdot \bar{u} = 0$. We thus look at solutions of type $(\mathbf{c},\mathbf{u}) = (\bar{\mathbf{c}}+\eps \tilde{\mathbf{c}}, \bar{\mathbf{u}}+\eps\tilde{\mathbf{u}})$, with $\eps\in (0,1]$ being the small parameter of the perturbation. The following theorem gathers the main properties that we are able to prove.

\smallskip
\begin{theorem}\label{theo:Cauchy MS}
Let $s> 3$ be an integer, $\func{\bar{u}}{\R_+\times\T^3}{\R^3}$ be in $L^\infty\pab{\R_+; H^{s}(\T^3)}$ with~${\nabla_x\cdot\bar{u}=0}$, and consider $\mathbf{\bar{c}} > 0$. There exist $\delta_\ms$, $C_\ms$, $C'_\ms$, $\lambda_\ms >0$ such that for all~${\eps\in (0,1]}$ and for any initial datum $(\tilde{\mathbf{c}}^\init, \tilde{\mathbf{u}}^\init)\in H^s(\T^3)\times H^{s-1}(\T^3)$ satisfying, for almost any $x\in\T^3$ and for any $1\leq i\leq N$,
\begin{itemize}
\item[(i)]\textbf{Mass compatibility: } $\disp{\sum_{i=1}^N \tilde{c}^\init_i(x) = 0 \quad\mbox{and}\quad \int_{\T^3}\tilde{c}_i^\init(x) \dd x = 0}$, \\[-0.5mm]
\item[(ii)]\textbf{Mass positivity: } $\disp{\bar{c}_i +\eps \tilde{c}_i^\init (x) > 0}$, \\[-1.5mm]
\item[(iii)]\textbf{Moment compatibility: } $\disp{\nabla_x \tilde{c}^\init_i = \sum_{j\neq i}\frac{\tilde{c}_i^\init \tilde{c}_j^\init}{\Delta_{ij}}\pa{\tilde{u}_j^\init - \tilde{u}_i^\init}}$,\\[-0.5mm]
\item[(iv)] \textbf{Smallness assumptions: } $\disp{\norm{\tilde{\mathbf{c}}^\init}_ {H^s_x}\leq \delta_\ms}\quad$ and $\quad\disp{\norm{\bar{u}}_{L^\infty_t H^{s}_x}\leq \delta_\ms}$,
\end{itemize}
there exists a unique weak solution 
$$\pa{\mathbf{c},\mathbf{u}} = \pab{\bar{\mathbf{c}}+\eps \tilde{\mathbf{c}}, \bar{\mathbf{u}}+\eps\tilde{\mathbf{u}}}$$
in $L^\infty\pab{\R_+;H^s(\T^3)} \times L^\infty\pab{\R_+;H^{s-1}(\T^3)}$ to the incompressible Maxwell-Stefan system \eqref{MS mass}--\eqref{MS momentum}--\eqref{MS incompressibility}, such that initially~${\restriction{\pa{\tilde{\mathbf{c}},\tilde{\mathbf{u}}}}{t=0} = \pa{\tilde{\mathbf{c}}^\init,\tilde{\mathbf{u}}^\init}}$ a.e. on~$\T^3$.

Moreover, $\mathbf{c}$ is positive and the following equimolar diffusion-like relation holds a.e. on~${\R_+\times\T^3}$:
\begin{equation}\label{equimolar vectorial}
\scalprod{\mathbf{c},\tilde{\mathbf{u}}}=\sum_{i=1}^N c_i(t,x) \tilde{u}_i(t,x)=0.
\end{equation}
Finally, for almost any time $t\geq 0$
\begin{eqnarray*}
\norm{\tilde{\mathbf{c}}}_{\sobolevx{s}} &\leq&   e^{- t \lambda_\ms}\norm{\tilde{\mathbf{c}}^\init}_{\sobolevx{s}},
\\[4mm]    \norm{\tilde{\mathbf{u}}}_{H^{s-1}_x} &\leq&  C_\ms e^{- t \lambda_\ms}\norm{\tilde{\mathbf{c}}^\init}_{\sobolevx{s}},
\\[3mm]    \int_0^t e^{2 (t-\tau) \lambda_\ms}\norm{\tilde{\mathbf{u}}(\tau)}^2_{H^{s}_x}\dd \tau &\leq & C'_\ms \norm{\tilde{\mathbf{c}}^\init}^2_{\sobolevx{s}}.
\end{eqnarray*}
The constants $\delta_\ms$, $\lambda_\ms$, $C_\ms$ and $C'_\ms$ are constructive and only depend on $s$, the number of species $N$, the diffusion coefficients $(\Delta_{ij})_{1\leq i,j\leq N}$ and the constant vector~$\mathbf{\bar{c}}$. In particular, they are independent of the parameter $\eps$.
\end{theorem}

\smallskip
\begin{remark}
Let us make a few comments about the above theorem.
\begin{enumerate}
\item Our analysis is actually independent of the parameter $\eps$ and we shall systematically bound it by $1$ in the estimates. However, we have decided to keep it because it recalls the perturbative regime (depending on the Knudsen number $\eps$) which is required at the kinetic level to rigorously derive \cite{BonBri} the Maxwell-Stefan system studied here.
\item The ``mass compatibility'' and the ``moment compatibility'' assumptions are not closure hypotheses, they actually exactly come from the system of equations \eqref{MS mass}--\eqref{MS momentum} applied at time $t=0$. We impose these conditions at the beginning, so that our initial datum is compatible with the Maxwell-Stefan system.
\item We emphasize that we only prove uniqueness in the perturbative setting under consideration.
\item The solution we construct has actually more regularity with respect to $t$, provided that $\bar{u}\in C^0\pab{\R_+;H^{s}(\T^3)}$. Indeed, in this case the couple $(\mathbf{c},\mathbf{u})$ also belongs to $C^0\pab{\R_+;H^{s-1}(\T^3)}\times C^0\pab{\R_+;H^{s-2}(\T^3)}$ with at least $\partial_t \tilde{\mathbf{c}} \in C^0\pab{\R_+; H^1(\T^3)}$. This allows in particular to properly define the initial value problem and give strong solutions. Going further, in a more general setting \cite{CheJun} the convection velocity $\bar{\mathbf{u}}$ is solution of an incompressible Navier-Stokes system with $\bar{u} \in C^0(\R_+; H^s(\T^3)) \cap C^1(\R_+; H^{s-2}(\T^3)) \cap L^2(\R_+; H^{s+1}(\T^3))$. We believe that in this case similar arguments used to prove the regularity of solutions to the Navier-Stokes equation \cite{MajBer} could be adapted to our problem and lead to recover the expected regularity for quasilinear parabolic systems \cite{Gio1}, namely $\tilde{\mathbf{c}} \in C^0(\R_+; H^s(\T^3)) \cap C^1(\R_+; H^{s-2}(\T^3)) \cap L^2(\R_+; H^{s+1}(\T^3))$ and $\tilde{\mathbf{u}} \in C^0(\R_+; H^{s-1}(\T^3)) \cap L^2(\R_+; H^s(\T^3))$.
\item The constants $\delta_{\mbox{\tiny{MS}}}$, $\lambda_{\mbox{\tiny{MS}}}$ and $C_{\mbox{\tiny{MS}}}$ are not explicitly computed, but their values can be determined respectively from formulae $\eqref{deltas}$, $\eqref{lambdas}$ and $\eqref{Cs}$. Note that the smallness $\delta_{\mbox{\tiny{MS}}}$ essentially depends on $\min\bar{c}_i$ so that once the constant state $\mathbf{\bar{c}}$ is chosen, it subjugates $\bar{u}$.
\item We emphasize that the exponential decay is due in part to the dynamics of the Maxwell-Stefan system and in part because Sobolev and Poincar\'e inequalities hold on the torus as it is a compact $C^1$-manifold.
\item At last, the positivity of the solution stems from the perturbative setting and does not follow a general weak minimum principle which can fail for cross-diffusion systems. It however seems that for the specific Maxwell-Stefan system under consideration, such a result could hold even in a non-perturbative setting \cite[Lemma 7.3.4]{Gio1}.
\end{enumerate}
\end{remark}
\bigskip

\section{Properties of the Maxwell-Stefan matrix}\label{sec:MS matrix}

\noindent We prove some properties of the Maxwell-Stefan matrix $A$, as well as some estimates on its derivatives. We conclude with properties and estimates on the pseudo-inverse of $A$ on its image. 

\smallskip
\begin{prop}\label{prop:spectral gap MS matrix}
For any $\mathbf{c} \geq 0$ the matrix $A(\mathbf{c})$ is nonpositive, in the sense that there exist two positive constants $\lambda_A$ and $\mu_A$ such that, for any $\mathbf{X}\in\R^N$,
\begin{gather*}
\norm{A(\mathbf{c})\mathbf{X}} \leq \mu_A \langle \mathbf{c},\mathbf{1}\rangle^2\norm{\mathbf{X}},
\\[3mm]     \langle \mathbf{X} , A(\mathbf{c})\mathbf{X} \rangle \leq - \lambda_A\pa{\min\limits_{1\leq i \leq N} c_i}^2\Big[\norm{\mathbf{X}}^2 -\langle\mathbf{X},\mathbf{1}\rangle^2\Big] \leq 0.
\end{gather*}
\end{prop}

\smallskip
\begin{proof}[Proof of Proposition \ref{prop:spectral gap MS matrix}]
Let us consider two $N$-vectors, $\mathbf{c} \geq 0$ and $\mathbf{X}$. The boundedness of $A(\mathbf{c})$ can be showed in the supremum norm, since all norms are equivalent in $\R^N$. It is straightforward that, for any $1\leq i\leq N,$
$$\abs{\sum\limits_{j=1}^N\frac{c_ic_j}{\Delta_{ij}}\pa{X_j-X_i}} \leq \frac{2\max\limits_{1\leq i \leq N} c_i}{\min\limits_{1\leq i,j \leq N}\Delta_{ij}} \pa{\sum\limits_{j=1}^N c_j}\max\limits_{1\leq j\leq N}\abs{X_j},$$
which raises the first inequality, since $\max\limits_{1\leq i\leq N} c_i \leq \sum_{j=1}^N c_j$ and we can thus choose $\mu_A = 2 \sqrt{N} / \min\limits_{1\leq i,j\leq N} \Delta_{ij}$ using the equivalence between the supremum norm and the standard Euclidean norm. 

\smallskip
We then compute
$$\langle \mathbf{X}, A(\mathbf{c})\mathbf{X} \rangle = - \sum\limits_{i=1}^N\sum\limits_{j=1}^N\frac{c_ic_j}{\Delta_{ij}}(X_i-X_j)X_i = - \frac{1}{2}\sum\limits_{i=1}^N\sum\limits_{j=1}^N\frac{c_ic_j}{\Delta_{ij}}(X_i-X_j)^2 \leq 0.$$
Note in particular that, in the case $c_i > 0$ and $\Delta_{ij}>0$ for any $1\leq i,j\leq N$, the relation $A(\mathbf{c})\mathbf{X} = 0$ implies $X_i = X_j$ for all $i$ and $j$, and so $\Ker A = \Span\pa{\mathbf{1}}$. If we now set $\lambda_A = 1/\max\limits_{1\leq i,j\leq N} \Delta_{ij}$, we can deduce the bound
$$\langle \mathbf{X}, A(\mathbf{c})\mathbf{X} \rangle \leq - \frac{\lambda_A}{2}\pa{\min\limits_{1\leq i \leq N} c_i}^2\sum\limits_{i=1}^N\sum\limits_{j=1}^N(X_i-X_j)^2,$$
and conclude the proof.
\end{proof}

As we shall need controls in Sobolev spaces, we then give below some estimates on the $x$-derivatives of the Maxwell-Stefan matrix.

\begin{prop}\label{prop:estimate derivative MS matrix}
Consider a multi-index $\alpha\in\N^3$ and let $\mathbf{c},\mathbf{U}\in H^{|\alpha|}(\T^3)$, with $\mathbf{c}\geq 0$. Then, for any $\mathbf{X}\in\R^N$,
\begin{equation*}
\begin{split}
\langle \partial^{\alpha}_x \cro{A(\mathbf{c})\mathbf{U}},\mathbf{X}\rangle \leq &\ \langle A(\mathbf{c}) \partial^{\alpha}_x\mathbf{U},\mathbf{X}\rangle  + 2 N^2 \mu_A \scalprod{\mathbf{c},\mathbf{1}}\norm{\mathbf{X}}  3^{\abs{\alpha}} \sum\limits_{\underset{|\alpha_1|\geq 1}{\alpha_1+\alpha_3 = \alpha}}\norm{\partial^{\alpha_1}_x\mathbf{c}}\norm{\partial^{\alpha_3}_x \mathbf{U}}
\\[1mm]    &\ + N^2\mu_A \norm{\mathbf{X}} 3^{\abs{\alpha}} \sum\limits_{\underset{|\alpha_1|, |\alpha_2| \geq 1}{\alpha_1+\alpha_2+\alpha_3 = \alpha}}\norm{\partial^{\alpha_1}_x \mathbf{c}}\norm{ \partial^{\alpha_2}_x \mathbf{c}}\norm{\partial^{\alpha_3}_x \mathbf{U}},
\end{split}
\end{equation*}
where $\mu_A$ is defined in Proposition \ref{prop:spectral gap MS matrix}.
\end{prop}

\smallskip
\begin{proof}[Proof of Proposition \ref{prop:estimate derivative MS matrix}]
Let $\mathbf{X}$ be in $\R^N$. We can explicitly compute
\begin{equation*}
\begin{split}
\langle \partial^{\alpha}_x\cro{A(\mathbf{c})\mathbf{U}},\mathbf{X}\rangle = & \sum\limits_{i=1}^N \partial^\alpha_x\pa{\sum\limits_{j=1}^N \frac{c_ic_j}{\Delta_{ij}}(U_j-U_i)}X_i
\\[2mm] = & \sum\limits_{1\leq i,j\leq N}X_i \sum\limits_{\alpha_1+\alpha_2+\alpha_3 = \alpha} \binom{\alpha}{\alpha_1,\alpha_2,\alpha_3}\frac{\partial^{\alpha_1}_x c_i \partial^{\alpha_2}_x c_j}{\Delta_{ij}}\pa{\partial^{\alpha_3}_x U_j - \partial^{\alpha_3}_x U_i}
\\[5mm] = & \ \langle A(\mathbf{c})\partial^{\alpha}_x\mathbf{U},\mathbf{X}\rangle 
\\[1mm]  & + \sum\limits_{1\leq i,j\leq N}X_i \sum\limits_{\underset{|\alpha_1|, |\alpha_2| \geq 1}{\alpha_1+\alpha_2+\alpha_3 = \alpha}}\binom{\alpha}{\alpha_1,\alpha_2,\alpha_3}\frac{\partial^{\alpha_1}_x c_i \partial^{\alpha_2}_x c_j}{\Delta_{ij}}\pa{\partial^{\alpha_3}_x U_j - \partial^{\alpha_3}_x U_i}
\\  & + \sum\limits_{1\leq i,j\leq N}X_i \sum\limits_{\underset{|\alpha_2|\geq 1}{\alpha_2+\alpha_3 = \alpha}}\binom{\alpha}{0,\alpha_2,\alpha_3}\frac{c_i \partial^{\alpha_2}_x c_j}{\Delta_{ij}}\pa{\partial^{\alpha_3}_x U_j - \partial^{\alpha_3}_x U_i} 
\\  & + \sum\limits_{1\leq i,j\leq N}X_i \sum\limits_{\underset{|\alpha_1|\geq 1}{\alpha_1+\alpha_3 = \alpha}}\binom{\alpha}{\alpha_1,0,\alpha_3}\frac{c_j \partial^{\alpha_1}_x c_i }{\Delta_{ij}}\pa{\partial^{\alpha_3}_x U_j - \partial^{\alpha_3}_x U_i}.
\end{split}
\end{equation*}
where $\binom{\alpha}{\alpha_1,\alpha_2,\alpha_3}$ is the multinomial coefficient which is bounded by $3^{\abs{\alpha}}$.

\smallskip
\par We then use Cauchy-Schwarz inequality and the fact that $\Delta = \min_{i,j} \Delta_{ij} >0$, together with $0 \leq c_i \leq \sum_{j=1}^N c_j$ and $\abs{c_i} \leq \norm{\mathbf{c}}$, to finally get
\begin{equation*}
\begin{split}
\langle \partial^{\alpha}_x\cro{A(\mathbf{c})\mathbf{U}},\mathbf{X}\rangle \leq & \ \langle A(\mathbf{c})\partial^{\alpha}_x\mathbf{U},\mathbf{X}\rangle 
\\[1mm]  & + \frac{4 N^2 3^{\abs{\alpha}}}{\Delta}\cro{\pa{\sum\limits_{j=1}^Nc_j} \sum\limits_{\underset{|\alpha_1|\geq 1}{\alpha_1+\alpha_2 = \alpha}}\norm{\partial^{\alpha_1}_x\mathbf{c}}\norm{\partial^{\alpha_2}_x \mathbf{U}}}\norm{\mathbf{X}}       
\\[1mm] & + \frac{2 N^2 3^{\abs{\alpha}}}{\Delta}\cro{\sum\limits_{\underset{|\alpha_1|, |\alpha_2| \geq 1}{\alpha_1+\alpha_2+\alpha_3 = \alpha}}\norm{\partial^{\alpha_1}_x \mathbf{c}}\norm{ \partial^{\alpha_2}_x \mathbf{c}}\norm{\partial^{\alpha_3}_x \mathbf{U}}}\norm{\mathbf{X}},
\end{split}
\end{equation*}
and the expected result follows.
\end{proof}

\smallskip
We conclude the present section with a control on the pseudoinverse of $A(\mathbf{c})$, which is defined on $\pab{\Span (\mathbf{1})}^\perp$ and nullspace $\mbox{Span}(\mathbf{1})$. We wish to stress again the fact that, contrary to \cite{CheJun,JunSte} where an entropy method is used, our analysis does not need the explicit expression of the pseudoinverse $A^{-1}$ to be carried out, which rather simplifies the computations.

\begin{prop}\label{prop:A-1}
For any $\mathbf{c}\in\pa{\R_+^*}^N$ and any $\mathbf{U}\in\pab{\Span (\mathbf{1})}^\bot$, the following estimates hold:
\begin{equation*}
\begin{split}
\norm{A(\mathbf{c})^{-1}\mathbf{U}} &\leq \frac{1}{\lambda_A \pa{\min\limits_{1\leq i \leq N} c_i }^2} \norm{\mathbf{U}},
\\[3mm]  \langle A(\mathbf{c})^{-1}\mathbf{U} , \mathbf{U} \rangle &\leq -\frac{\lambda_A \pa{\min\limits_{1\leq i \leq N} c_i }^2}{\mu_A^2\langle \mathbf{c},\mathbf{1}\rangle^4}\norm{\mathbf{U}}^2,
\end{split}
\end{equation*}
where $\lambda_A$ and $\mu_A$ are defined in Proposition \ref{prop:spectral gap MS matrix}.
\end{prop}

\smallskip
\begin{proof}[Proof of Proposition \ref{prop:A-1}]
The proof is a direct application of Proposition \ref{prop:spectral gap MS matrix}. Indeed, Cauchy-Schwarz inequality yields, for any $\mathbf{X}\in\pab{\Span (\mathbf{1})}^\bot$,
$$-\norm{\mathbf{X}}\norm{A(\mathbf{c})\mathbf{X}} \leq -\lambda_A \pa{\min\limits_{1\leq i \leq N} c_i}^2 \norm{\mathbf{X}}^2,$$
so that

$$\norm{\mathbf{X}} \leq \frac{1}{\lambda_A \pa{\min\limits_{1\leq i \leq N} c_i}^2}\norm{A(\mathbf{c})\mathbf{X}},\hspace{4cm}$$
which proves the first estimate by simply taking $\mathbf{X} = A(\mathbf{c})^{-1}\mathbf{U}$.

\medskip
The spectral gap property comes from the boundedness of $A(\mathbf{c})$, given by Proposition \ref{prop:spectral gap MS matrix} for $\mathbf{X} = A(\mathbf{c})^{-1}\mathbf{U}$, which translates into a coercivity estimate
$$\norm{\mathbf{U}}^2 \leq \pabb{\mu_A\langle \mathbf{c},\mathbf{1}\rangle^2}^2 \norm{A(\mathbf{c})^{-1}(\mathbf{U})}^2,$$
that we plug into the spectral gap inequality satisfied by $A(\mathbf{c})$.
\end{proof}
\bigskip

\section{Perturbative Cauchy theory for the Maxwell-Stefan system}\label{sec:perturbed IMS}

\noindent We recall the vectorial form of Maxwell-Stefan system \eqref{MS mass}--\eqref{MS incompressibility}:
\begin{gather}
\partial_t\mathbf{c}+ \nabla_x\cdot\pa{\mathbf{c}\mathbf{u}} = 0, \label{MS mass vectorial}
\\[3mm]  \nabla_x \mathbf{c} = A(\mathbf{c})\mathbf{u}, \label{MS momentum vectorial}
\\[3.5mm]  \nabla_x \cdot\langle \mathbf{c}, \mathbf{u} \rangle =0. \label{MS incompressibility vectorial}
\end{gather}
where
$$A\pa{\mathbf{c}} = \pa{\frac{c_i c_j}{\Delta_{ij}} - \pa{\sum\limits_{k=1}^{N}\frac{c_i c_k}{\Delta_{ik}}}\delta_{ij}}_{1\leq i,j \leq N}.$$
The Cauchy theory we build offers an explicit description of all the solutions $(\mathbf{c},\mathbf{u})$ which are perturbed around a global macroscopic equilibrium state. We point out in particular that, because of the incompressibility condition $\eqref{MS incompressibility vectorial}$, any macroscopic state for which $\bar{\mathbf{c}}$ is stationary has the form~${(\bar{c}_i,\bar{u})_{1\leq i\leq N}}$, where each $\bar{c}_i$ is a positive constant and the velocity $\bar{u}:\R_+\times \T^3\to\R^3$, common to all the species, satisfies $\nabla_x\cdot \bar{u}(t,x)=0$ for any $t\geq 0$ and $x\in\T^3$. For a sake of clarity, throughout the present section any perturbative vector-valued function~${\mathbf{w}=(w_1,\ldots, w_N)}$ shall be written under the specific form $\mathbf{w}=\bar{\mathbf{w}}+\eps \tilde{\mathbf{w}}$, where the component $\bar{\mathbf{w}}$ with the overbar symbol always refers to some (macroscopic) stationary state of the mixture and the component $\tilde{\mathbf{w}}$ overlined by a tilde refers to the fluctuation around the corresponding equilibrium state. Moreover, note that for simplicity the specific quantity $\bar{\mathbf{u}}$ will always denote an $N$-vector where all the components are given by a common incompressible velocity~$\bar{u}$.

\smallskip
The present section is divided into two parts. In the first one, we show how to derive the new orthogonal system equivalent to \eqref{MS mass vectorial}--\eqref{MS momentum vectorial}--\eqref{MS incompressibility vectorial}, and state the counterpart of Theorem \ref{theo:Cauchy MS} in terms of this new reformulation for the unknowns $\mathbf{c}$ and $\mathbf{U}$, the orthogonal part of $\mathbf{u}$. In the second part we prove all the required properties~(existence and uniqueness, positivity and exponential decay to equilibrium) for the couple $(\mathbf{c},\mathbf{U})$, properties that will be also satisfied by the original unknowns $(\mathbf{c},\mathbf{u})$.

\subsection{An orthogonal incompressible Maxwell-Stefan system}\label{subsec:orthogonal MS}  

Here we present the equivalent orthogonal reformulation of \eqref{MS mass vectorial}--\eqref{MS momentum vectorial}--\eqref{MS incompressibility vectorial}, which allows to transfer the study of existence and uniqueness for solutions~$(\mathbf{c},\mathbf{u})$ to the development of a Cauchy theory for the new unknowns $(\mathbf{c},\mathbf{U})$, where we denote with~${\mathbf{U}=\mathbf{u}-\pi_A(\mathbf{u})}$ the part of $\mathbf{u}$ that is projected onto $(\Span(\mathbf{1}))^\perp$, $\pi_A$ being the orthogonal projection onto $\Ker A = \Span(\mathbf{1})$. 
\par Before stating our result, we introduce a useful notation that allows to preserve the vectorial structure of the Maxwell-Stefan system. We suppose that, for some $V\in\R^3$ and some $N$-vector $\mathbf{w}\in(\R^3)^N$ whose components lie in $\R^3$, the standard notation of the scalar product in $\R^3$ is extended to any multiplication of type $V\cdot\mathbf{w}$ in the following sense
\begin{equation*}
V\cdot\mathbf{w} = (V\cdot w_i)_{1\leq i\leq N}.
\end{equation*}

\medskip
\begin{prop}\label{prop:MS orthogonal writing}
Let $s\in\N^*,\  C_0>0$, and consider two functions $\mathbf{c}^\inits \in H^s(\T^3)$ and~${\mathbf{u}^\inits\in H^{s-1}(\T^3)}$ verifying, for almost any $x\in\T^3$,
$$\mathbf{c}^\inits (x)\geq 0 \quad  \textrm{ and }\quad \sum\limits_{i=1}^N c_i^\inits (x)=C_{0}.$$
Then, $(\mathbf{c},\mathbf{u})\in L^\infty\pab{\R_+;H^s(\T^3)}\times L^\infty \pab{\R_+;H^{s-1}(\T^3})$ is a solution to the incompressible Maxwell-Stefan system \eqref{MS mass vectorial}--\eqref{MS momentum vectorial}--\eqref{MS incompressibility vectorial}, associated to the initial datum $(\mathbf{c}^\inits,\mathbf{u}^\inits)$, if and only if there exist two functions $\mathbf{U}:\R_+\times\T^3\to\R^{3N}$ and \ $\bar{u}:\R_+\times\T^3\to \R^3$ in~${L^\infty\pab{\R_+; H^{s-1}(\T^3})}$ such that, for almost any~$(t,x)\in\R_+\times\T^3,$
\begin{gather}
\mathbf{U}(t,x) \in \pa{\Span(\mathbf{1})}^\perp \quad \textrm{ and }\quad \nabla_x\cdot \bar{u}(t,x)=0, \label{Condition1}
\\[3mm]    \hspace{-0.8cm}  \mathbf{u}(t,x) = \mathbf{\bar{u}}(t,x) + \mathbf{V}_U(t,x) \quad \textrm{ with } \quad \mathbf{V}_U = \mathbf{U} - \frac{1}{C_0}\langle \mathbf{c},\mathbf{U}\rangle\mathbf{1}, \label{Condition2}
\\[3mm]      \left\{\begin{array}{l}  \partial_t \mathbf{c} +  \nabla_x\cdot\pa{\mathbf{c}\mathbf{V}_U} +  \bar{u} \cdot \nabla_x \mathbf{c} = 0, \\[5mm]    \nabla_x\mathbf{c} = A(\mathbf{c})\mathbf{U}. \end{array}\right. \label{Condition3}
\end{gather}
\end{prop}

\medskip
\begin{remark}
The above result is not difficult to prove but we underline again that it is of great importance, since it turns the incompressible Maxwell-Stefan system \eqref{MS mass vectorial}--\eqref{MS momentum vectorial}--\eqref{MS incompressibility vectorial} with full velocity vectors $\mathbf{u}$ into a system only depending on their orthogonal component $\mathbf{U}\in \pab{\Span (\mathbf{1})}^\perp$, while the projection onto $\Span (\mathbf{1})$ raises a simple transport term in the continuity equation $\eqref{MS mass vectorial}$. Notice in particular that we differentiate between $C_0=\scalprod{\mathbf{c}^\inits,\mathbf{1}}$ in $\eqref{Condition2}$ and $\langle \mathbf{c},\mathbf{1}\rangle$ in $\eqref{Condition3}$. As we shall see, in both equations it will turn out that these two quantities are equal, but keeping the notation $\langle \mathbf{c},\mathbf{1}\rangle$ offers a fully closed system.

\smallskip
Moreover, Proposition \ref{prop:MS orthogonal writing} actually shows that all perturbative solutions of the Maxwell-Stefan system \eqref{MS mass}--\eqref{MS momentum}--\eqref{MS incompressibility} are of the form described by Theorem \ref{theo:Cauchy MS}, that is $\bar{\mathbf{c}}+\eps \tilde{\mathbf{c}}$ and $\bar{\mathbf{u}}+\eps\tilde{\mathbf{u}}$.  Note however that Theorem \ref{theo:Cauchy MS} is not optimal, since we require $\bar{\mathbf{u}}$ to be more regular than the perturbation $\tilde{\mathbf{u}}$.
\end{remark}

\begin{proof}[Proof of Proposition \ref{prop:MS orthogonal writing}]
Let $\pi_A$ be the orthogonal projection operator onto $\mbox{Span} (\mathbf{1})$ and consider a solution $(\mathbf{c},\mathbf{u})$ of the Maxwell-Stefan system \eqref{MS mass vectorial}--\eqref{MS momentum vectorial}--\eqref{MS incompressibility vectorial}. The first implication directly follows from the decomposition
\begin{equation*}\label{decomposition}
\mathbf{u} =  \pi_A(\mathbf{u}) + \pab{\mathbf{u}-\pi_A(\mathbf{u})} =  \frac{\langle \mathbf{u},\mathbf{1} \rangle}{\norm{\mathbf{1}}^2} \mathbf{1} + \pi_A^\bot(\mathbf{u}),
\end{equation*}
where we recall that $\norm{\cdot}$ defines the Euclidean norm induced by the scalar product $\scalprod{\cdot,\cdot}$ in~$\R^N$, weighted by the vector $\mathbf{1}$.

First of all, observe that summing over the continuity equations $\eqref{MS mass vectorial}$ and using the incompressibility condition $\eqref{MS incompressibility vectorial}$, it follows that $\partial_t \langle \mathbf{c},\mathbf{1}\rangle = 0$. Moreover, if we sum the gradient relations $\eqref{MS momentum vectorial}$, we also get
$$\nabla_x\pa{\sum\limits_{i=1}^N c_i} = \sum\limits_{1\leq i,j \leq N} \frac{c_ic_j}{\Delta_{ij}}\pa{u_j-u_i} = 0.$$
Therefore, the quantity $\langle \mathbf{c},\mathbf{1}\rangle$ is independent of $t$ and $x$, allowing to initially deduce that
\begin{equation}\label{constant total mass}
\sum\limits_{i=1}^N c_i(t,x) = \sum\limits_{i=1}^N c^\init_i(x)=C_0\quad \textrm{a.e. on } \R_+\times \T^3.
\end{equation}
Now, defining $\mathbf{U} = \pi_A^\bot(\mathbf{u})$ and $W = \frac{\langle \mathbf{u},\mathbf{1}^2 \rangle}{\norm{\mathbf{1}}}$, we easily recover \eqref{Condition1}--\eqref{Condition2}. The transport equation $\eqref{MS momentum vectorial}$ can then be rewritten in terms of $\mathbf{U}$ and $W$ as
\begin{equation}\label{start transport orthogonal}
\partial_t \mathbf{c} + \nabla_x\cdot\pa{\mathbf{c}\mathbf{U}} + \mathbf{c}\nabla_x\cdot W + W\cdot \nabla_x\mathbf{c} =0.
\end{equation}
In a similar way, the incompressibility condition $\eqref{MS incompressibility vectorial}$ in these new unknowns reads
\begin{eqnarray*}
0 &=& \sum\limits_{i=1}^N\nabla_x\cdot \pa{c_i\pa{U_i + W}} = \sum\limits_{i=1}^N\nabla_x\cdot \pa{c_i U_i} + \nabla_x\pa{\sum\limits_{i=1}^Nc_i}\cdot W+ \pa{\nabla_x\cdot W}\sum\limits_{i=1}^N c_i
\\&=& \nabla_x\cdot{\langle \mathbf{c},\mathbf{U}\rangle} + C_0\nabla_x\cdot W,
\end{eqnarray*}
where we have used $\eqref{constant total mass}$. We thus infer the existence of a divergence-free function~${\func{\bar{u}}{\R_+\times\T^3}{\R^3}}$ such that, for almost any $(t,x)\in\T^3\times\R^3,$
\begin{equation} \label{eq:convection velocity}
\left\{\begin{array}{l} 
\nabla_x\cdot \bar{u}(t,x) = 0, 
\\[3mm]     W(t,x) = -\frac{1}{C_0}\langle \mathbf{c},\mathbf{U}\rangle (t,x) + \bar{u}(t,x). 
\end{array} \right.
\end{equation}
Plugging the above relation into $\eqref{start transport orthogonal}$ and replacing $C_0$ by its value $\langle \mathbf{c},\mathbf{1}\rangle$, we recover the first equation of $\eqref{Condition3}$. Finally, the decomposition $\eqref{decomposition}$ also yields the second relation of $\eqref{Condition3}$, since $\pi_A(\mathbf{u})\in \Ker A$ and thus
$$A(\mathbf{c})\mathbf{u} = A(\mathbf{c})\mathbf{U},$$
proving that $(\mathbf{c},\mathbf{U},\bar{u})$ is a solution of the orthogonal reformulation \eqref{Condition1}--\eqref{Condition2}--\eqref{Condition3}.

\medskip
Consider now a triple $(\mathbf{c},\mathbf{U},\bar{u})$ satisfying conditions \eqref{Condition1}--\eqref{Condition2}--\eqref{Condition3}. The reverse implication then follows by defining
\begin{equation}\label{orthogonal solution}
\mathbf{u} = \mathbf{U} + \pa{- \frac{1}{C_0}\langle \mathbf{c},\mathbf{U}\rangle + \bar{u}}\mathbf{1}.
\end{equation}
Indeed, summing over $1\leq i\leq N$ the gradient relations of $\eqref{Condition3}$, we get
$$\nabla_x\pa{\sum\limits_{i=1}^N c_i} =0,$$
which is used when one also sums over $1\leq i\leq N$ the transport equations of $\eqref{Condition3}$, to deduce
\begin{equation*}
0 = \partial_t\pa{\sum\limits_{i=1}^N c_i} + \nabla_x \pa{\sum\limits_{i=1}^N c_i U_i} - \nabla_x\cdot\pa{ \langle \mathbf{c},\mathbf{U}\rangle\frac{\pa{\sum\limits_{i=1}^N c_i}}{\langle \mathbf{c}, \mathbf{1}\rangle}}= \partial_t \pa{\sum\limits_{i=1}^N c_i},
\end{equation*}
since $\scalprod{\mathbf{c},\mathbf{1}}=\sum_{i=1}^N c_i$ by definition.
Thus, the quantity $\scalprod{\mathbf{c},\mathbf{1}}$ is independent of $(t,x)$, allowing to infer that
$$\sum\limits_{i=1}^N c_i(t,x) = \scalprod{\mathbf{c}^\inits,\mathbf{1}} = C_0,\quad \textrm{a.e. on }\   \R_+\times \T^3.$$
This recovery of $\eqref{constant total mass}$ not only implies the incompressibility condition $\eqref{MS incompressibility vectorial}$ but also, with the divergence free property of $\bar{u}$, that
\begin{eqnarray*}
\nabla_x\cdot\pa{\mathbf{c}\mathbf{u}}&=& \nabla_x\cdot\pa{\mathbf{c}\pa{\mathbf{U}-\frac{\langle \mathbf{c},\mathbf{U}\rangle}{C_0} \mathbf{1}}} + \bar{u}\cdot \nabla_x \mathbf{c}
\\&=&\nabla_x\cdot\pa{\mathbf{c}\pa{\mathbf{U}-\frac{\langle \mathbf{c},\mathbf{U}\rangle}{\langle \mathbf{c},\mathbf{1}\rangle} \mathbf{1}}} + \bar{u}\cdot \nabla_x \mathbf{c}.
\end{eqnarray*}
Therefore, the first equation of $\eqref{Condition3}$ rewrites
$$\partial_t\mathbf{c} + \nabla_x\cdot\pa{\mathbf{c}\mathbf{u}}=0,$$
and, thanks again to the fact that $\Ker A=\Span(\mathbf{1})$, one finally sees that
$$\nabla_x\mathbf{c} = A(\mathbf{c})\mathbf{U} = A(\mathbf{c})\mathbf{u}.$$
This ensures that $(\mathbf{c},\mathbf{u})$, with $\mathbf{u}$ defined by $\eqref{orthogonal solution}$, solves the Maxwell-Stefan system \eqref{MS mass vectorial}--\eqref{MS momentum vectorial}--\eqref{MS incompressibility vectorial}, thus concluding the proof.

\end{proof}

\begin{remark}
Since the divergence free component $\bar{u}$ must solve equation $\eqref{eq:convection velocity}$, easy computations using the definitions of $\pi_A(\mathbf{u})$ and $\mathbf{U}$ show that~$\bar{u} = \frac{\langle \mathbf{c}, \mathbf{u} \rangle}{\langle \mathbf{c}, \mathbf{1} \rangle}$, which coincides with the molar average velocity of the mixture, \textit{i.e.} the convection velocity. Moreover, the reconstruction condition $\eqref{Condition2}$ tells us that the full velocity $\mathbf{u}$ is obtained from $\mathbf{\bar{u}}$ and from the vector $\mathbf{V}_U$, which satisfies the equimolar relation $\langle \mathbf{c}, \mathbf{V}_U \rangle = 0$ and thus corresponds to the diffusion velocity. This feature has been pointed out in the introduction: the kinetic decomposition of $\mathbf{u} = \pi_A(\mathbf{u}) + \pi_A(\mathbf{u})^\perp$ into macroscopic and microscopic part, combined with the incompressibility condition that we have imposed on the total flux, naturally leads to the physical splitting of $\mathbf{u}$ into convection velocities $\bar{\mathbf{u}}$ and diffusion velocities $\mathbf{V}_U$.
\end{remark}

By means of this orthogonal reformulation, we can now prove our main result.

\subsection{Proof of Theorem \ref{theo:Cauchy MS}}\label{subsec:Cauchy MS}  

This last part is devoted to showing the validity of Theorem \ref{theo:Cauchy MS}. We shall divide the proof into several steps which help in enlightening the basic ideas behind our strategy. We first restate our result about solutions $(\mathbf{c},\mathbf{u})$ in terms of the orthogonal reformulation \eqref{Condition1}--\eqref{Condition2}--\eqref{Condition3}, about solutions $(\mathbf{c},\mathbf{U},\bar{u})$. Thanks to preliminary lemmata describing the main properties of the matrix $A$ and of its pseudo-inverse obtained in Section \ref{sec:MS matrix} - with range $\pa{\mbox{Span}(\mathbf{1})}^\bot$ and nullspace $\mbox{Span}(\mathbf{1})$, we then derive uniform (in~$\eps$) \textit{a priori} energy estimates for the solution $(\mathbf{c},\mathbf{U})$, which provide the exponential relaxation towards the global equilibrium~$(\bar{\mathbf{c}},\bar{\mathbf{u}})$. Starting from this we are thus able to recover the positivity of $\mathbf{c}$, and to prove global existence and uniqueness for solutions to $\eqref{Condition3}$ having the specific perturbative forms $\mathbf{c}=\bar{\mathbf{c}}+\eps\tilde{\mathbf{c}}$ and $\mathbf{U}=\eps\tilde{\mathbf{U}}$. The combination of these results will eventually allow to deduce global existence, uniqueness and exponential decay for the couple $(\mathbf{c},\mathbf{u})$, using the reconstruction condition $\eqref{Condition2}$.

\bigskip
\noindent \textbf{Step 1 -- Reformulation in terms of orthogonal velocities.} Let us begin with a simple lemma needed in order to understand the shape of the velocities $\mathbf{U}$ and $\bar{u}$, when they are associated to a constant state $\bar{\mathbf{c}}$.

\medskip
\begin{lemma}\label{prop:stationary orthogonal}
Let $s\in \N^*$ and let $\bar{\mathbf{c}}$ be a positive constant $N$-vector. For any functions~${\mathbf{U},\bar{u}}$ in~$L^\infty\pab{\R_+; H^{s-1}(\T^3)}$ such that $\mathbf{U}\in\pab{\Span (\mathbf{1})}^\perp$ and $\nabla_x\cdot \bar{u}=0$, a triple $(\bar{\mathbf{c}},\mathbf{U},\bar{u})$ is solution to the system of equations \eqref{Condition1}--\eqref{Condition2}--\eqref{Condition3} if and only if
$$\mathbf{U}(t,x) =0\quad \textrm{a.e. on }\ \R_+\times \T^3.$$
\end{lemma}

\begin{proof}[Proof of Lemma \ref{prop:stationary orthogonal}]
The proof is very simple. Because $\bar{\mathbf{c}}$ is constant, the gradient equation of $\eqref{Condition3}$ reads $A(\bar{\mathbf{c}})\mathbf{U}=0$. But $\mathbf{U}$ belongs to $(\Span(\mathbf{1}))^\perp$, which means that the pseudoinverse $A^{-1}$ - with range $\pa{\mbox{Span}(\mathbf{1})}^\bot$ and nullspace $\mbox{Span}(\mathbf{1})$ -  remains well-defined. Consequently, for almost any $(t,x)\in\R_+\times\T^3$, $\mathbf{U}(t,x)=0$. 

\smallskip
The reverse implication is direct.
\end{proof}

We are now interested in building a Cauchy theory for the orthogonal form of the Maxwell-Stefan system, around the stationary solutions  given by Lemma \ref{prop:stationary orthogonal}. More precisely, we want to prove existence and uniqueness for perturbative solutions to \eqref{Condition1}--\eqref{Condition2}--\eqref{Condition3} of the form
$$\left\{\begin{array}{l}\mathbf{c}(t,x) = \bar{\mathbf{c}} + \eps \tilde{\mathbf{c}},\\[2mm] \mathbf{U} = \eps\tilde{\mathbf{U}}.\end{array}\right.$$
In terms of these particular solutions, the system \eqref{Condition1}--\eqref{Condition2}--\eqref{Condition3} translates into
\begin{eqnarray}
&&\partial_t \tilde{\mathbf{c}} + \bar{\mathbf{c}}\nabla_x\cdot \mathbf{V}_{\tilde{U}} + \bar{u}\cdot\nabla_x\tilde{\mathbf{c}}+ \eps \nabla_x\cdot\pa{\tilde{\mathbf{c}}\mathbf{V}_{\tilde{U}}}=0, \label{perturbed MS mass}
\\[2mm]  &&\nabla_x\tilde{\mathbf{c}} = A(\bar{\mathbf{c}}+\eps\tilde{\mathbf{c}})\tilde{\mathbf{U}}, \label{perturbed MS momentum}
\end{eqnarray}
with the notation $\mathbf{V}_{\tilde{U}} = \tilde{\mathbf{U}} - \frac{\langle \mathbf{\mathbf{c}},\tilde{\mathbf{U}}\rangle}{\langle \mathbf{\mathbf{c}},\mathbf{1}\rangle}\mathbf{1}$.
The orthogonal reformulation of Theorem \ref{theo:Cauchy MS} then writes in the following way.

\smallskip
\begin{theorem}\label{theo:Cauchy MS orthogonal}
Let $s > 3$ be an integer, $\func{\bar{u}}{\R_+\times\T^3}{\R^3}$ be in $L^\infty\pab{\R_+;H^{s}(\T^3)}$ with $\nabla_x\cdot \bar{u}=0$ and consider a constant $N$-vector $\bar{\mathbf{c}}>0$. There exist $\delta_s$, $C_s$, $\lambda_s >0$ such that for all $\eps\in (0,1]$ and for any~${\pa{\tilde{\mathbf{c}}^\init,\tilde{\mathbf{U}}^\init}}\in H^s(\T^3) \times H^{s-1}(\T^3)$ satisfying, for almost any $x\in\T^3$ and for any~${1\leq i\leq N}$,
\begin{itemize}
\item[(i)]\textbf{Mass compatibility: } $\disp{\sum_{i=1}^N \tilde{c}^\init_i(x) = 0 \quad\mbox{and}\quad \int_{\T^3}\tilde{c}_i^\init(x) \dd x = 0}$,\\
\item[(ii)]\textbf{Mass positivity: } $\disp{\bar{c}_i +\eps \tilde{c}_i^\init(x) > 0}$, \\
\item[(iii)]\textbf{Moment compatibility: } $\disp{\tilde{\mathbf{U}}^\init(x) = A\pa{\bar{\mathbf{c}}+\eps\tilde{\mathbf{c}}^\init(x)}^{-1}\nabla_x \tilde{\mathbf{c}}^\init(x)}$,\\
\item[(iv)] \textbf{Smallness assumptions: } $\disp{\norm{\tilde{\mathbf{c}}^\init}_{H^s_x}\leq \delta_s}\quad$ and $\quad\disp{\norm{\bar{u}}_{L^\infty_t H^{s}_x}\leq \delta_s}$,
\end{itemize}
there exists a unique weak solution $\pa{\tilde{\mathbf{c}},\tilde{\mathbf{U}}}\in L^\infty\pab{\R_+;H^s(\T^3)} \times L^\infty\pab{\R_+;H^{s-1}(\T^3)}$ to the system of equations \eqref{perturbed MS mass}--\eqref{perturbed MS momentum}, having $\pa{\tilde{\mathbf{c}}^\init, \tilde{\mathbf{U}}^\init}$ as initial datum. In particular, for almost any $(t,x)\in\R_+\times\T^3$, the vector $\mathbf{c}(t,x)=\bar{\mathbf{c}}+\eps\tilde{\mathbf{c}}(t,x)$ is positive and~$\tilde{\mathbf{U}}(t,x)$ belongs to $(\Span(\mathbf{1}))^\perp$.

\noindent Moreover, the following estimates hold for almost any $t\geq 0$
\begin{eqnarray*}
\norm{\tilde{\mathbf{c}}}_{\sobolevx{s}}  & \leq & e^{-\lambda_s t}\norm{\tilde{\mathbf{c}}^\init}_{\sobolevx{s}},
\\[4mm]     \norm{\tilde{\mathbf{U}}}_{H^{s-1}_x} & \leq & C_s e^{-\lambda_s t}\norm{\tilde{\mathbf{c}}^\init}_{\sobolevx{s}},
\\[4mm]        \int_0^te^{2\lambda_s(t-\tau)}\norm{\tilde{\mathbf{U}}(\tau)}^2_{H^{s}_x}\dd \tau & \leq &  C'_s \norm{\tilde{\mathbf{c}}^\init}^2_{\sobolevx{s}}.
\end{eqnarray*}
The constants $\delta_s$, $\lambda_s$ and $C_s$ are constructive and are given respectively by $\eqref{deltas}$, $\eqref{lambdas}$ and $\eqref{Cs}$.
\end{theorem}
\bigskip


\medskip
\noindent \textbf{Step 2 -- \textit{A priori} energy estimates and positivity.} The two \textit{a priori} results~(exponential decay and positivity of $\mathbf{c}$) that we now derive are of crucial importance, as they will allow us to exhibit existence and uniqueness for the couple $(\tilde{\mathbf{c}},\tilde{\mathbf{U}})$ in the next section.
\par Before we start, we present a simple result which establishes two relevant properties satisfied by the solution of \eqref{perturbed MS mass}--\eqref{perturbed MS momentum}. We show in particular that $\tilde{\mathbf{c}}$ has zero mean on the torus, a feature that will let us exploit Poincar\'e inequality in the proof of the \textit{a priori} estimates. 

\smallskip
\begin{lemma}\label{lem:preservations}
Let $\bar{\mathbf{c}}, C_0>0$ be such that $\scalprod{\bar{\mathbf{c}},\mathbf{1}}=C_0$, and consider a triple $(\tilde{\mathbf{c}}^\init,\tilde{\mathbf{U}}^\init,\bar{u})$ satisfying the hypotheses of Theorem \ref{theo:Cauchy MS orthogonal}. If $(\tilde{\mathbf{c}},\tilde{\mathbf{U}})$ is a weak solution of \eqref{perturbed MS mass}--\eqref{perturbed MS momentum} with initial datum $(\tilde{\mathbf{c}}^\init,\tilde{\mathbf{U}}^\init)$, then, for almost any $(t,x)\in\R_+\times\T^3$ and for any $1\leq i\leq N$,
\begin{gather}
\sum\limits_{i=1}^N \tilde{c}_i(t,x) = 0  \quad \textrm{ and }\quad  \int_{\T^3} \tilde{c}_i(t,x)\dd x = 0.
\end{gather}
In particular, the conservation of the total mass $\langle \mathbf{c} ,\mathbf{1}\rangle = C_0$ holds almost everywhere on~${\R_+\times\T^3}$.
\end{lemma}

\smallskip
\begin{proof}[Proof of Lemma \ref{lem:preservations}]
We have already showed how to recover the preservation of the total mass inside the proof of Proposition \ref{prop:MS orthogonal writing}. 

\vskip 1mm
The second property follows directly from the fact that $\bar{\mathbf{c}}$ is a constant $N$-vector and $\bar{u}$ is divergence-free. Indeed, using these two assumptions the mass equation $\eqref{perturbed MS mass}$ can be written under a divergent form as
$$\partial_t \tilde{\mathbf{c}} + \nabla_x \cdot \pa{\mathbf{c}\VV+\tilde{\mathbf{c}}\bar{\mathbf{u}}} =0.$$
Integrating over the torus we thus obtain
$$\frac{\dd}{\dd t}\int_{\T^3} \tilde{\mathbf{c}}(t,x)\dd x = 0\quad \textrm{for a.e. } t \geq 0,$$
which gives the expected result since $\tilde{\mathbf{c}}^\init$ has zero mean on the torus.
\end{proof}

\medskip
The result providing the \textit{a priori} energy estimates is then the following.

\smallskip
\begin{prop}\label{prop:a priori MS}
Let $s>3$ be an integer. There exist $\delta_s, \lambda_s, C_s , C'_s >0$ such that, under the assumptions of Theorem \ref{theo:Cauchy MS orthogonal}, if $(\tilde{\mathbf{c}},\tilde{\mathbf{U}},\bar{u})$ is a solution of the perturbed orthogonal system \eqref{perturbed MS mass}--\eqref{perturbed MS momentum} satisfying the initial controls
$$\norm{\tilde{\mathbf{c}}^\init}_{H^s_x}\leq \delta_s\ \quad\mbox{and}\quad\ \norm{\bar{u}}_{L^\infty_t H^{s}_x}\leq \delta_s,$$
then, for almost any $t\geq 0,$
\begin{eqnarray*}
\norm{\tilde{\mathbf{c}}}_{\sobolevx{s}} & \leq &  e^{-\lambda_s t}\norm{\tilde{\mathbf{c}}^\init}_{\sobolevx{s}},
\\[4mm]     \norm{\tilde{\mathbf{U}}}_{H^{s-1}_x} & \leq & C_s e^{-\lambda_s t}\norm{\tilde{\mathbf{c}}^\init}_{\sobolevx{s}},
\\[4mm]        \int_0^te^{2\lambda_s(t-\tau)}\norm{\tilde{\mathbf{U}}(\tau)}^2_{H^{s}_x}\dd \tau & \leq & C'_s \norm{\tilde{\mathbf{c}}^\init}^2_{\sobolevx{s}}.
\end{eqnarray*}
The constants $\delta_s$, $\lambda_s$, $C_s$ and $C'_s$ are explicit and only depend on $s$, the number of species $N$, the diffusion coefficients $(\Delta_{ij})_{1\leq i,j\leq N}$ and the constant vector~$\bar{\mathbf{c}}$. In particular, they are independent of the parameter $\eps$.
\end{prop}

\smallskip
\begin{proof}[Proof of Proposition \ref{prop:a priori MS}]
We fix a multi-index $\alpha\in\N^3$ such that $\abs{\alpha} \leq s$. Recall that we have defined
$$\VV = \tilde{\mathbf{U}} - \frac{\langle \mathbf{\mathbf{c}},\tilde{\mathbf{U}}\rangle}{\langle \mathbf{\mathbf{c}},\mathbf{1}\rangle}\mathbf{1}.$$
We successively apply the $\alpha$-derivative to the transport equation $\eqref{perturbed MS mass}$, take the scalar product with the vector $\disp\pa{\frac{1}{\bar{c}_i}\partial^\alpha_x \tilde{c}_i}_{1\leq i \leq N}$, and integrate over $\T^3$. This yields, after integrating by parts,
\begin{equation}\label{a priori MS start}
\begin{split}
\frac{1}{2}\frac{\dd}{\dd t}\norm{\partial^\alpha_x\tilde{\mathbf{c}}}^2_{L^2_x\big(\mathbf{\bar{c}}^{-\frac{1}{2}}\big)} =& \int_{\T^3}\langle \nabla_x\partial^\alpha_x \tilde{\mathbf{c}} , \partial^\alpha_x\VV \rangle \dd x + \int_{\T^3} \langle \nabla_x\partial^\alpha_x \tilde{\mathbf{c}} , \partial^\alpha_x\pa{\tilde{\mathbf{c}}\bar{\mathbf{u}}} \rangle_{\mathbf{\bar{c}}^{-1}} \dd x 
\\[1mm]          & + \eps \int_{\T^3}\langle \nabla_x\partial^\alpha_x \tilde{\mathbf{c}} , \partial^\alpha_x\pa{\tilde{\mathbf{c}}\VV} \rangle_{\mathbf{\bar{c}}^{-1}} \dd x.
\end{split}
\end{equation}
We estimate these three terms separately. We first notice that summing over $i$ the gradient equations $\eqref{perturbed MS momentum}$ we obtain
$$\sum\limits_{i=1}^N \nabla_x \tilde{c}_i(t,x) = 0\quad \textrm{a.e. on }\ \R_+\times\T^3,$$
which means that $\nabla_x \tilde{\mathbf{c}}$ belongs to $\big(\Span(\mathbf{1})\big)^\perp$. Applying the $\alpha$-derivative to both sides of this relation then gives
$$\sum\limits_{i=1}^N \nabla_x \partial^\alpha_x \tilde{c}_i(t,x) = 0 \quad \textrm{a.e. on }\ \R_+\times\T^3,$$
from which we deduce that also
\begin{equation}\label{orthogonal higher derivative c}
\nabla_x\partial^\alpha_x\tilde{\mathbf{c}} \in \pab{\mbox{Span}(\mathbf{1})}^\bot \quad \textrm{a.e. on }\ \R_+\times\T^3.
\end{equation}

Thanks to the orthogonality $\eqref{orthogonal higher derivative c}$ of the higher derivatives and using the gradient relation $\eqref{perturbed MS momentum}$, the first term on the right-hand side of $\eqref{a priori MS start}$ becomes
\begin{equation*}
\begin{split}
\int_{\T^3}\langle \nabla_x\partial^\alpha_x \tilde{\mathbf{c}} , \partial^\alpha_x\VV \rangle \dd x & = \int_{\T^3}\langle \partial^\alpha_x\nabla_x \tilde{\mathbf{c}} , \partial^\alpha_x\tilde{\mathbf{U}} \rangle \dd x -  \frac{1}{C_0}  \int_{\T^3} \langle \nabla_x\partial^\alpha_x \tilde{\mathbf{c}} , \mathbf{1} \rangle \partial^\alpha_x \langle\mathbf{c},\tilde{\mathbf{U}}\rangle \dd x
\\[2mm] & = \int_{\T^3}\langle \partial^\alpha_x [A(\mathbf{c})\tilde{\mathbf{U}}] , \partial^\alpha_x\tilde{\mathbf{U}} \rangle \dd x.
\end{split}
\end{equation*}
We apply Proposition \ref{prop:estimate derivative MS matrix} with $\mathbf{X}=\partial^\alpha_x \tilde{\mathbf{U}}$ and use the mass conservation (Lemma \ref{lem:preservations}) and the spectral gap of $A(\mathbf{c})$ (Proposition \ref{prop:spectral gap MS matrix}) to recover the initial bound
\begin{equation*}
\begin{split}
\int_{\T^3} \langle  \nabla_x\partial^\alpha_x \tilde{\mathbf{c}},\partial^\alpha_x \VV \rangle \dd x \leq & \int_{\T^3} \langle A(\mathbf{c})\partial^\alpha_x\tilde{\mathbf{U}},\partial^\alpha_x\tilde{\mathbf{U}} \rangle \dd x
\\[5mm]  & + 2 N^2 \mu_A C_0 3^s \int_{\T^3} \norm{\partial^\alpha_x\tilde{\mathbf{U}}} \sum\limits_{\underset{|\alpha_1|\geq 1}{\alpha_1+\alpha_3 = \alpha}}\norm{\partial^{\alpha_1}_x\mathbf{c}}\norm{\partial^{\alpha_3}_x \tilde{\mathbf{U}}} \dd x
\\[3mm]  &  + N^2\mu_A 3^s \int_{\T^3}  \norm{\partial^\alpha_x\tilde{\mathbf{U}}} \sum\limits_{\underset{|\alpha_1|, |\alpha_2| \geq 1}{\alpha_1+\alpha_2 + \alpha_3 = \alpha}}\norm{\partial^{\alpha_1}_x \mathbf{c}}\norm{ \partial^{\alpha_2}_x \mathbf{c}}\norm{\partial^{\alpha_3}_x \tilde{\mathbf{U}}} \dd x
\end{split}
\end{equation*}
\begin{equation*}
\begin{split}
\leq & - \lambda_A\pa{\min\limits_{1\leq i \leq N}c_i}^2\norm{\partial^\alpha_x\tilde{\mathbf{U}}}_{L^2_x}^2
\\  & + 2 \eps N^2 \mu_A C_0 3^s \norm{\tilde{\mathbf{U}}}_{H^s_x}\cro{\int_{\T^3} \pa{\sum\limits_{\underset{|\alpha_1|\geq 1}{\alpha_1+\alpha_3 = \alpha}}\norm{\partial^{\alpha_1}_x\tilde{\mathbf{c}}}\norm{\partial^{\alpha_3}_x \tilde{\mathbf{U}}}}^2 \dd x}^{\frac{1}{2}}\ 
\\   & + \eps^2 N^2\mu_A 3^s \norm{\tilde{\mathbf{U}}}_{H^s_x}\cro{\int_{\T^3} \pa{\sum\limits_{\underset{|\alpha_1|, |\alpha_2| \geq 1}{\alpha_1+\alpha_2+\alpha_3 = \alpha}}\norm{\partial^{\alpha_1}_x \tilde{\mathbf{c}}}\norm{ \partial^{\alpha_2}_x \tilde{\mathbf{c}}}\norm{\partial^{\alpha_3}_x \tilde{\mathbf{U}}}}^2 \dd x}^{\frac{1}{2}},
\end{split}
\end{equation*}
where we have also used the Cauchy-Schwarz inequality and the fact that the $L^2_x$ norm of $\partial^\alpha_x\tilde{\mathbf{U}}$ is controlled by the $H^s_x$ norm of $\tilde{\mathbf{U}}$.
\par Recalling our choice $s>3$, in order to control the bi and tri-norm terms inside the integrals we use the continuous embedding of $H^{s/2}_x$ in $L^\infty_x$, which holds as soon as $s/2 > 3/2$. We detail our procedure for the tri-norm term, the bi-norm term being treated in the same way. Since $\alpha_1+\alpha_2+\alpha_3 = \alpha$, at most one of the $|\alpha_i|$ can be strictly larger than $\abs{\alpha}/2$. Hence, at least two $\abs{\alpha_i}$ are lower or equal to $\abs{\alpha}/2 \leq s/2$. We therefore split the tri-norm term into three sums as
\begin{equation*} 
\begin{split}
\pa{\sum\limits_{\underset{|\alpha_1|,|\alpha_2| \geq 1}{\alpha_1+\alpha_2+\alpha_3 = \alpha}}\norm{\partial^{\alpha_1}_x \tilde{\mathbf{c}}}\norm{ \partial^{\alpha_2}_x \tilde{\mathbf{c}}} \norm{\partial^{\alpha_3}_x \tilde{\mathbf{U}}}}^2 \leq  & \ \sum\limits_{\underset{\underset{|\alpha_1|, |\alpha_2| \leq \frac{s}{2}}{|\alpha_1|, |\alpha_2| \geq 1}}{\alpha_1+\alpha_2+\alpha_3 = \alpha}} \norm{\partial^{\alpha_1}_x \tilde{\mathbf{c}}}^2\norm{ \partial^{\alpha_2}_x \tilde{\mathbf{c}}}^2\norm{\partial^{\alpha_3}_x \tilde{\mathbf{U}}}^2  
\\[2mm] & + \sum\limits_{\underset{\underset{|\alpha_1|, |\alpha_3| \leq \frac{s}{2}}{|\alpha_1|, |\alpha_2| \geq 1}}{\alpha_1+\alpha_2+\alpha_3 = \alpha}} \norm{\partial^{\alpha_1}_x \tilde{\mathbf{c}}}^2\norm{ \partial^{\alpha_2}_x \tilde{\mathbf{c}}}^2\norm{\partial^{\alpha_3}_x \tilde{\mathbf{U}}}^2
\\[2mm]   & +   \sum\limits_{\underset{\underset{|\alpha_2|, |\alpha_3| \leq \frac{s}{2}}{|\alpha_1|, |\alpha_2| \geq 1}}{\alpha_1+\alpha_2+\alpha_3 = \alpha}} \norm{\partial^{\alpha_1}_x \tilde{\mathbf{c}}}^2\norm{ \partial^{\alpha_2}_x \tilde{\mathbf{c}}}^2\norm{\partial^{\alpha_3}_x \tilde{\mathbf{U}}}^2.
\end{split}
\end{equation*}
For any $\alpha_k$-derivative such that $|\alpha_k|\leq s/2$ we bound the corresponding factor by its $L^\infty_x$ norm and we exploit the mentioned embedding of $H^{s/2}_x$ in $L^\infty_x$ in order to recover the correct Sobolev norm. In the sequel $C_{\rm sob}$ will refer to any positive constant that appears when using the Sobolev embeddings. When integration over $\T^3$ is considered, the first sum then produces
\begin{equation*}
\begin{split}
\sum\limits_{\underset{\underset{|\alpha_1|, |\alpha_2| \leq \frac{s}{2}}{|\alpha_1|, |\alpha_2| \geq 1}}{\alpha_1+\alpha_2+\alpha_3 = \alpha}}\int_{\T^3}\norm{\partial^{\alpha_1}_x \tilde{\mathbf{c}}}^2 & \norm{ \partial^{\alpha_2}_x \tilde{\mathbf{c}}}^2    \norm{\partial^{\alpha_3}_x \tilde{\mathbf{U}}}^2 \dd x 
\\[-2mm]     & \leq  \sum\limits_{\underset{\underset{|\alpha_1|, |\alpha_2| \leq \frac{s}{2}}{|\alpha_1|, |\alpha_2| \geq 1}}{\alpha_1+\alpha_2+\alpha_3 = \alpha}}\norm{\partial^{\alpha_1}_x \tilde{\mathbf{c}}}_{L^\infty_x}^2\norm{ \partial^{\alpha_2}_x \tilde{\mathbf{c}}}_{L^\infty_x}^2\norm{\partial^{\alpha_3}_x \tilde{\mathbf{U}}}_{L^2_x}^2
\\[1mm]    &  \leq C_{\rm sob}\sum\limits_{\underset{\underset{|\alpha_1|, |\alpha_2| \leq \frac{s}{2}}{|\alpha_1|, |\alpha_2| \geq 1}}{\alpha_1+\alpha_2+\alpha_3 = \alpha}}\norm{\partial^{\alpha_1}_x \tilde{\mathbf{c}}}_{H^{s/2}_x}^2\norm{ \partial^{\alpha_2}_x \tilde{\mathbf{c}}}_{H^{s/2}_x}^2\norm{\partial^{\alpha_3}_x \tilde{\mathbf{U}}}_{L^2_x}^2
\\[1mm]     & \leq s^3 C_{\rm sob} \norm{\tilde{\mathbf{c}}}_{H^s_x}^4 \norm{\tilde{\mathbf{U}}}_{H^s_x}^2,
\end{split}
\end{equation*}
and the two others are dealt with in the same way. Consequently, the tri-norm term can be estimated as
\begin{equation} \label{Sobolev embedding}
\int_{\T^3} \pa{\sum\limits_{\underset{|\alpha_1|,|\alpha_2| \geq 1}{\alpha_1+\alpha_2+\alpha_3 = \alpha}}\norm{\partial^{\alpha_1}_x \tilde{\mathbf{c}}}\norm{ \partial^{\alpha_2}_x \tilde{\mathbf{c}}} \norm{\partial^{\alpha_3}_x \tilde{\mathbf{U}}}}^2 \dd x \leq 3s^3 C_{\rm sob} \norm{\tilde{\mathbf{c}}}_{H^s_x}^4\norm{\tilde{\mathbf{U}}}^2_{H^s_x}.
\end{equation}
Moreover, the previous Sobolev embedding also yields, for any $1\leq i\leq N$,
\begin{equation}\label{min c tilde}
c_i(t,x) \geq \bar{c}_i - \eps \norm{\tilde{\mathbf{c}}}_{L^\infty_x} \geq \min_{1\leq i\leq N} \bar{c}_i - \eps C_{\rm sob} \norm{\tilde{\mathbf{c}}}_{H^s_x}\quad \textrm{a.e. on }\ \R_+\times\T^3.
\end{equation}
We thus infer the first upper bound
\begin{equation}\label{a priori MS 1/3}
\begin{split}
\int_{\T^3}\langle \nabla_x\partial^\alpha_x \tilde{\mathbf{c}},\partial^\alpha_x \VV \rangle \dd x  \leq  -\lambda_A\Bigg(\min\limits_{1\leq i \leq N} & \bar{c}_i -  \eps C_{\rm sob} \norm{\tilde{\mathbf{c}}}_{H^s_x}\Bigg)^2 \norm{\partial^\alpha_x \tilde{\mathbf{U}}}^2_{L^2_x}
\\[2mm]    & + \eps 3^s s^2 N^2 \mu_A \pa{4 C_0 + 6 \eps \norm{\tilde{\mathbf{c}}}_{H^s_x}}\norm{\tilde{\mathbf{c}}}_{H^s_x}\norm{\tilde{\mathbf{U}}}_{H^{s}_x}^2.
\end{split}
\end{equation}

The second and third term on the right-hand side of $\eqref{a priori MS start}$ are handled more easily. As we did for establishing $\eqref{Sobolev embedding}$, we apply the Leibniz derivation rule and the Cauchy-Schwarz inequality, together with the Sobolev embedding that allows to distribute the $H^s_x$ norm to each factor of the products. In this way we obtain the estimates 
\begin{equation*}
\begin{split}
\int_{\T^3} \langle \nabla_x\partial^\alpha_x \tilde{\mathbf{c}}, \partial^\alpha_x\pa{\tilde{\mathbf{c}}\bar{\mathbf{u}}} \rangle_{\mathbf{\bar{c}}^{-1}} \dd x &\leq \frac{C_{\rm sob}}{\min\limits_{1\leq i\leq N}\bar{c}_i}\norm{\nabla_x\partial^\alpha_x\tilde{\mathbf{c}}}_{L^2_x}\norm{\tilde{\mathbf{c}}}_{H^s_x}\norm{\bar{\mathbf{u}}}_{H^s_x},
\\[2mm]      \int_{\T^3}\langle \nabla_x\partial^\alpha_x \tilde{\mathbf{c}}, \partial^\alpha_x\pa{\tilde{\mathbf{c}}\VV} \rangle_{\mathbf{\bar{c}}^{-1}} \dd x &\leq \frac{C_{\rm sob}}{\min\limits_{1\leq i\leq N}\bar{c}_i} \norm{\nabla_x\partial^\alpha_x\tilde{\mathbf{c}}}_{L^2_x}\norm{\tilde{\mathbf{c}}}_{H^s_x}\pa{1+\frac{1}{C_0}\norm{\mathbf{c}}_{H^s_x}}\norm{\tilde{\mathbf{U}}}_{H^s_x},
\end{split}
\end{equation*}
where we have used that $\frac{1}{\bar{c}_i}\leq \frac{1}{\min_{i} \bar{c}_i}$ for any $1\leq i\leq N$. In order to control the $L^2_x$ norm of $\nabla_x \partial^\alpha_x\tilde{\mathbf{c}}$, we exploit the gradient relation $\eqref{perturbed MS momentum}$. By similar computations to the ones providing the estimate of Proposition \ref{prop:estimate derivative MS matrix}, and thanks to the continuous Sobolev embedding~${H^{s/2}_x\hookrightarrow L^\infty_x}$, one infers
\begin{eqnarray}
\norm{\nabla_x\partial^\alpha_x\tilde{\mathbf{c}}}_{L^2_x}\leq \norm{\nabla_x\tilde{\mathbf{c}}}_{H^s_x} & = & \norm{A(\mathbf{c})\tilde{\mathbf{U}}}_{H^s_x} \nonumber
\\[2mm]    &  \leq  &  C_{\rm sob}(s^2+C_0s)\norm{\mathbf{c}}^2_{H^s_x} \norm{\tilde{\mathbf{U}}}_{H^s_x} \nonumber
\\[2mm]    &  \leq  &  2C_{\rm sob}(s^2+C_0s)\pa{C_0^2\abs{\T^3} + \eps^2\norm{\tilde{\mathbf{c}}}^2_{H^s_x}}\norm{\tilde{\mathbf{U}}}_{H^s_x} \nonumber
\\[2mm]   &  \leq & C_s\pa{1 + \eps^2\norm{\tilde{\mathbf{c}}}^2_{H^s_x}}\norm{\tilde{\mathbf{U}}}_{H^s_x},  \label{estimate c by U}
\end{eqnarray}
thanks to the basic estimates
\begin{gather*}
\int_{\T^3} \pa{\bar{c}_i + \eps \tilde{c}_i}^2\dd x  \leq 2\pa{ \int_{\T^3}\bar{c}_i^{\hspace{0.5mm }2} \dd x + \eps^2\int_{\T^3}\tilde{c}_i^{\hspace{0.5mm} 2}\dd x},
\\[3mm]  0 \leq \bar{c}_i \leq \sum\limits_{j=1}^N \bar{c}_j=C_0.
\end{gather*}
Since $\eps\leq 1$ and $\norm{\bar{\mathbf{u}}}_{H^s_x}\leq \sqrt{N} \norm{\bar{u}}_{H^s_x}$, we finally deduce the upper bounds
\begin{align}
\int_{\T^3} \langle \nabla_x\partial^\alpha_x \tilde{\mathbf{c}}, \partial^\alpha_x\pa{\tilde{\mathbf{c}}\bar{\mathbf{u}}} \rangle_{\mathbf{\bar{c}}^{-1}} \dd x & \leq \frac{\sqrt{N} C_s}{\min\limits_{1\leq i\leq N}\bar{c}_i}\norm{\tilde{\mathbf{c}}}_{H^s_x}\pa{1+\norm{\tilde{\mathbf{c}}}_{H^s_x}}^2\norm{\bar{u}}_{H^s_x}\norm{\tilde{\mathbf{U}}}_{H^s_x},    \label{a priori MS 2/3}
\\[4mm]   \int_{\T^3}\langle \nabla_x\partial^\alpha_x \tilde{\mathbf{c}}, \partial^\alpha_x\pa{\tilde{\mathbf{c}}\VV} \rangle_{\mathbf{\bar{c}}^{-1}}\dd x & \leq \frac{C_s}{\min\limits_{1\leq i\leq N}\bar{c}_i}\norm{\tilde{\mathbf{c}}}_{H^s_x}\pa{1+\norm{\tilde{\mathbf{c}}}_{H^s_x}}^3\norm{\tilde{\mathbf{U}}}^2_{H^s_x},   \label{a priori MS 3/3}
\end{align}
by accordingly increasing the value of the constant $C_s$.

\smallskip
To conclude, we gather $\eqref{a priori MS start}$ with the estimates $\eqref{a priori MS 1/3}$, $\eqref{a priori MS 2/3}$ and $\eqref{a priori MS 3/3}$, and we sum over all $\abs{\alpha}\leq s$. In this way, we obtain
\begin{equation*}
\begin{split}
\frac{1}{2}\frac{\dd}{\dd t} \norm{\tilde{\mathbf{c}}}^2_{\sobolevx{s}} \leq &\  - \lambda_A \pa{\min\limits_{1\leq i \leq N}\bar{c}_i}^2 \norm{\tilde{\mathbf{U}}}^2_{H^s_x} 
\\[2mm]  &\ + \frac{s^3 \sqrt{N} C_s}{\min\limits_{1\leq i\leq N}\bar{c}_i}\norm{\tilde{\mathbf{c}}}_{H^s_x}\norm{\bar{u}}_{H^s_x} \pa{1+\norm{\tilde{\mathbf{c}}}_{H^s_x}}^2 \norm{\tilde{\mathbf{U}}}_{H^s_x}
\\[1mm]   &\  + \eps \Bigg( 2 \lambda_A C_{\rm sob} \min_{1\leq i\leq N} \bar{c}_i + \eps \lambda_A C_{\rm sob}^2 +\frac{s^3 C_s}{\min\limits_{1\leq i\leq N}\bar{c}_i}  
\\[1mm] &  \qquad\quad + 3^s s^5 C_{\rm sob} N^2 \mu_A(4 C_0 + 6\eps) \Bigg) \norm{\tilde{\mathbf{c}}}_{H^s_x}\Big(1+\norm{\tilde{\mathbf{c}}}_{H^s_x}\Big)^3 \norm{\tilde{\mathbf{U}}}^2_{H^s_x}.
\end{split}
\end{equation*}
In order to close the estimate above, since $\bar{\mathbf{c}}$ is constant we first easily check that
$$\norm{\tilde{\mathbf{c}}}_{H^s_x} \leq \max\limits_{1\leq i \leq N}\bar{c}_i \norm{\tilde{\mathbf{c}}}_{\sobolevx{s}} \leq C_0 \norm{\tilde{\mathbf{c}}}_{\sobolevx{s}}.$$
Moreover, recalling Lemma \ref{lem:preservations}, we can apply the Poincar\'e inequality to $\tilde{\mathbf{c}}$, which has zero mean on the torus. Denoting $C_{\T^3}>0$ the Poincar\'e constant, we can thus compute
\begin{equation}\label{poincare}
\norm{\tilde{\mathbf{c}}}_{H^s_x} \leq C_{\T^3}\norm{\nabla_x\tilde{\mathbf{c}}}_{H^s_x} \leq C_{\T^3}C_s\pa{1+\eps^2\norm{\tilde{\mathbf{c}}}^2_{H^s_x}}\norm{\tilde{\mathbf{U}}}_{H^s_x},
\end{equation}
where we have also used $\eqref{estimate c by U}$.

\smallskip
We denote by $C_s$ any positive constant that only depends on $s$, $N$, $\lambda_A$, $\mu_A$, $\mathbf{\bar{c}}$, $C_{\rm sob}$, $C_0$ and $C_{\T^3}$. Thanks to the above estimates, we can consequently infer the validity of the bound

\begin{equation} \label{main estimate}
\begin{split}
\frac{1}{2}\frac{\dd}{\dd t}\norm{\tilde{\mathbf{c}}}^2_{\sobolevx{s}}  \leq - \Bigg[ & \lambda_A \pa{\min\limits_{1\leq i \leq N}\bar{c}_i}^2  - C_s \norm{\bar{u}}_{H^s_x} \pa{1+\norm{\tilde{\mathbf{c}}}_{\sobolevx{s}}}^4
\\[1mm]   & \ \ \ \qquad\qquad - C_s \norm{\tilde{\mathbf{c}}}_{\sobolevx{s}}\pa{1+\norm{\tilde{\mathbf{c}}}_{\sobolevx{s}}}^3\Bigg]\norm{\tilde{\mathbf{U}}}^2_{H^s_x}.
\end{split}
\end{equation}
Therefore, if $\norm{\tilde{\mathbf{c}}^\init}_{\sobolevx{s}} \leq \delta_s$ and $\norm{\bar{u}}_{H^s_x} \leq \delta_s$ for almost any~${t\geq 0}$, where $\delta_s >0$ is chosen such that
\begin{equation}\label{deltas}
C_s\delta_s\pabb{(1+\delta_s)^4+(1+\delta_s)^3} \leq \frac{\lambda_A \pa{\min\limits_{1\leq i \leq N}\bar{c}_i}^2}{2},
\end{equation}
we ensure that the $\sobolevx{s}$ norm of $\tilde{\mathbf{c}}$ keeps diminishing and satisfies
\begin{equation}\label{a priori MS final}
\frac{\dd}{\dd t}\norm{\tilde{\mathbf{c}}}^2_{\sobolevx{s}} \leq - \lambda_A \pa{\min\limits_{1\leq i \leq N}\bar{c}_i}^2 \norm{\tilde{\mathbf{U}}}^2_{H^s_x}\quad \textrm{for a.e. } t\geq 0.
\end{equation}
Moreover, since the Poincar\'e inequality $\eqref{poincare}$ tells us that the norm of $\tilde{\mathbf{U}}$ controls the one of $\tilde{\mathbf{c}}$, we recover the estimate
\begin{equation*}
\frac{\dd}{\dd t}\norm{\tilde{\mathbf{c}}}^2_{\sobolevx{s}} \leq  -\frac{\lambda_A \pa{\min\limits_{1\leq i \leq N}\bar{c}_i}^2}{ C^2_{\T^3}C^2_s(1+\delta_s^2)^2}\norm{\tilde{\mathbf{c}}}^2_{H^s_x}  \leq   -\frac{\lambda_A  \pa{\min\limits_{1\leq i \leq N}\bar{c}_i}^3}{C^2_{\T^3}C^2_s(1+\delta_s^2)^2}\norm{\tilde{\mathbf{c}}}^2_{\sobolevx{s}}.
\end{equation*}
Setting
\begin{equation}\label{lambdas}
\lambda_s = \frac{\lambda_A \pa{\min\limits_{1\leq i \leq N}\bar{c}_i}^3}{2 C^2_{\T^3}C^2_s(1+\delta_s^2)^2},\hspace{5cm}
\end{equation}
Gr\"onwall's lemma finally tells us that for a.e. $t\geq 0$
$$\norm{\tilde{\mathbf{c}}}_{\sobolevx{s}} \leq  e^{-\lambda_s t}\norm{\tilde{\mathbf{c}}^\init}_{\sobolevx{s}},$$
and we also recover
\begin{equation}\label{Cs}
\norm{\tilde{\mathbf{U}}}_{H^{s-1}_x} = \norm{A(\mathbf{c})^{-1} \nabla_x\tilde{\mathbf{c}}}_{H^{s-1}_x} \leq \tilde{C}_s\norm{\nabla_x\tilde{\mathbf{c}}}_{H^{s-1}_x} \leq \tilde{C}_s\norm{\tilde{\mathbf{c}}}_{H^{s}_x}\leq C_s\norm{\tilde{\mathbf{c}}}_{\sobolevx{s}},
\end{equation}
by simply adjusting the value of $C_s$. The constant $C_s = C_s(C_0,\lambda_A,\mu_A,s,\delta_s,\bar{\mathbf{c}})>0$ is obtained by inverting $A(\mathbf{c})$ and repeating the previous computations, \textit{via} the continuous Sobolev embedding already mentioned. In particular, note that for our choice of $\delta_s$ one sees from $\eqref{min c tilde}$ that $\mathbf{c}$ does not vanish anywhere and there is therefore no singularity in $A(\mathbf{c})^{-1}$. 

\smallskip
The last estimate on the integral of $\norm{\tilde{\mathbf{U}}}^2_{H^s_x}$ is a direct application of Gr\"onwall's lemma from $\eqref{a priori MS final}$. This concludes the proof.
\end{proof}

\medskip
Before going into details in the proofs of existence and uniqueness, we present here another result which establishes that the positivity of $\mathbf{c}$ is obtained \textit{a priori}. This will help the reader in clarifying the last statement we gave in the previous proof, about the invertibility of $A(\mathbf{c})$. Moreover, note that ensuring the positivity of $\mathbf{c}$ \textit{a priori} is crucial, since it will leave us free on the choice of the iterative scheme to be used in the next section, when constructing the solution of the system \eqref{perturbed MS mass}--\eqref{perturbed MS momentum}.

\smallskip
\begin{lemma}\label{lem:positivity}
Consider an initial datum $(\tilde{\mathbf{c}}^\init,\tilde{\mathbf{U}}^\init)$ satisfying the assumptions of Theorem \ref{theo:Cauchy MS orthogonal} with $\delta_s$ sufficiently small.  If $(\tilde{\mathbf{c}},\tilde{\mathbf{U}})$ is a solution of \eqref{perturbed MS mass}--\eqref{perturbed MS momentum} with initial datum $(\tilde{\mathbf{c}}^\init,\tilde{\mathbf{U}}^\init)$, then, for almost any $(t,x)\in\R_+\times\T^3$ the vector $\bar{\mathbf{c}}+\eps\tilde{\mathbf{c}}(t,x)$ is positive.
\end{lemma}

\begin{proof}[Proof of Lemma \ref{lem:positivity}]
The proof follows straightforwardly in perturbative regime in regular Sobolev spaces. Indeed, the previous \textit{a priori} estimate and Sobolev embedding proves that
$$\norm{\tilde{\mathbf{c}}}_{L^\infty_{t,x}} \leq C_{Sob} \norm{\tilde{\mathbf{c}}^{in}}_{H^s_{x}} \leq C'_{Sob}\delta_s.$$
Therefore one could modify the definition of $\delta_s$ so that also
$$ C'_{Sob}\delta_s \leq \min\limits_{1\leq i \leq N} \frac{\bar{c}_i}{2} $$
and thus $\mathbf{c}(t,x)$ is positive a.e. on $\R_+\times\T^3$.
\end{proof} 

\medskip
\noindent \textbf{Step 3 -- Existence and uniqueness of the couple $(\tilde{\mathbf{c}},\tilde{\mathbf{U}})$.} We now have all the tools needed in order to construct our Cauchy theory for the couple $(\tilde{\mathbf{c}},\tilde{\mathbf{U}})$. We shall first present the existence result and then prove the uniqueness of the constructed solution. 

\begin{prop}\label{lem:existence MS}
Let $s>3$ be an integer and consider a triple $(\tilde{\mathbf{c}}^\init,\tilde{\mathbf{U}}^\init,\bar{u})$ satisfying the assumptions of Theorem \ref{theo:Cauchy MS orthogonal}. There exists $\delta_s>0$ such that, for all~${\eps\in (0,1]}$, there exists a global weak solution $(\tilde{\mathbf{c}},\tilde{\mathbf{U}})\in L^\infty\pab{\R_+; H^s(\T^3)}\times L^\infty\pab{\R_+; H^{s-1}(\T^3)}$ of the system \eqref{perturbed MS mass}--\eqref{perturbed MS momentum}, with initial datum $(\tilde{\mathbf{c}}^\init,\tilde{\mathbf{U}}^\init)$.
\end{prop}

\begin{proof}[Proof of Proposition \ref{lem:existence MS}]
The proof is standard and is based on an iterative scheme, where we first construct a solution on a well-chosen time interval $[0,T_0]$, and we then show that this interval can be extended to $[0,+\infty)$. Note however that one has to be careful with the estimates, since the conservation of the exact exponential decay rate is crucial. The underlying mechanism lies on the fact that our problem is actually quasilinear parabolic for small initial data. Indeed, noticing that 
$$ \tilde{\mathbf{U}} = A(\mathbf{c})^{-1}\nabla_x\tilde{\mathbf{c}}, $$
we solely have to solve
\begin{equation*}
\begin{split}
\partial_t \tilde{\mathbf{c}}  +  & \mathbf{\bar{c}}\nabla_x\cdot\pa{A(\mathbf{c})^{-1}\nabla_x\tilde{\mathbf{c}}- \frac{\langle \mathbf{c},A(\mathbf{c})^{-1}\nabla_x\tilde{\mathbf{c}} \rangle}{\langle \mathbf{\mathbf{c},\mathbf{1}\rangle}}\mathbf{1}} + \bar{u}\cdot \nabla_x\tilde{\mathbf{c}}
\\[1mm]   &  \quad\quad\quad+ \eps \nabla_x\cdot\pa{\tilde{\mathbf{c}}\pa{A(\mathbf{c})^{-1}\nabla_x\tilde{\mathbf{c}} - \frac{\langle \mathbf{c},A(\mathbf{c})^{-1}\nabla_x\tilde{\mathbf{c}} \rangle}{\langle \mathbf{c},\mathbf{1}\rangle}\mathbf{1}}} = 0.
\end{split}
\end{equation*}
From Proposition \ref{prop:A-1} we see that the higher order term is of order $2$, symmetric in a $\tilde{c}/\bar{c}$ reformulation and negative for $\mathbf{c} >0$, which makes us tackle this equation like a quasilinear parabolic one.

\smallskip
We initially set $\tilde{\mathbf{c}}^{(0)} = \tilde{\mathbf{c}}^\init $ and suppose that an $N$-vector function $\tilde{\mathbf{c}}^{(n)}$ belonging to $L^\infty\pab{0,T_0;H^s(\T^3)}$ is given, satisfying 
\begin{equation}\label{assumption iterative}
\norm{\tilde{\mathbf{c}}^{(n)}}_{\sobolevx{s}} \leq \delta_s e^{-\lambda_s t}, \qquad \sum\limits_{i=1}^N \tilde{c}^{\hspace{0.6mm}(n)}_i(t,x)=0\quad \textrm{a.e. on }\ (0,T_0)\times\T^3.
\end{equation}
For $s >3$, the Sobolev embedding $H^s_x \hookrightarrow L^\infty_x$ makes applicable standard hyperbolic-parabolic methods on the torus \cite{Kawashima84} which raise the existence of a solution $\tilde{\mathbf{c}}^{(n+1)}$ belonging to $L^2\pab{0,T_0;H^1(\T^3)}$ to the following linear equation
\begin{multline}\label{iterative scheme}
\partial_t \tilde{\mathbf{c}}^{(n+1)} + \bar{u}\cdot \nabla_x\tilde{\mathbf{c}}^{(n+1)}
\\[1mm]    + \nabla_x\cdot\pa{\mathbf{c}^{(n)}\pa{A(\mathbf{c}^{(n)})^{-1}\nabla_x\tilde{\mathbf{c}}^{(n+1)} - \frac{\langle \mathbf{c}^{(n)},A(\mathbf{c}^{(n)})^{-1} \nabla_x\tilde{\mathbf{c}}^{(n+1)} \rangle}{\langle \mathbf{c}^{(n)},\mathbf{1}\rangle}\mathbf{1}}} = 0,
\end{multline}
with initial datum $\tilde{\mathbf{c}}^\init$. Note that summing the above system of equations over $1\leq i\leq N$ yields 
$$\sum\limits_{i=1}^N \tilde{c}^{\hspace{0.5mm}(n+1)}_i(t,x)= \sum\limits_{i=1}^N \tilde{c}^{\hspace{0.5mm}(n+1)}_i(0,x) = 0\quad \textrm{a.e. on }\ (0,T_0)\times\T^3,$$
which shows, thanks to Proposition \ref{prop:A-1}, that $\pab{\mbox{Span}(\mathbf{1})}^\bot$ is stable for $\eqref{iterative scheme}$ and so the $N$-vector function $A(\mathbf{c}^{(n)})^{-1}\nabla_x\tilde{\mathbf{c}}^{(n+1)}$ is well-defined for almost every time $t\in (0,T_0)$.

\smallskip
If we now define $\tilde{\mathbf{U}}^{(n+1)} = A(\mathbf{c}^{(n)})^{-1}\nabla_x\tilde{\mathbf{c}}^{(n+1)}$, the specific choice of our time-discretisation $\eqref{iterative scheme}$ allows to follow the computations carried out to obtain the \textit{a priori} estimates in Proposition \ref{prop:a priori MS} and leads to the equivalent of inequality $\eqref{main estimate}$, which reads
\begin{multline*}
\frac{1}{2}\frac{\dd}{\dd t}\norm{\tilde{\mathbf{c}}^{(n+1)}}^2_{\sobolevx{s}}  \leq - \Bigg[ \lambda_A \pa{\min\limits_{1\leq i \leq N}\bar{c}_i}^2  - C_s \norm{\bar{u}}_{H^s_x} \pa{1+\norm{\tilde{\mathbf{c}}^{(n)}}_{\sobolevx{s}}}^4
\\[1mm]  - C_s \norm{\tilde{\mathbf{c}}^{(n)}}_{\sobolevx{s}}\pa{1+\norm{\tilde{\mathbf{c}}^{(n)}}_{\sobolevx{s}}}^3\Bigg]\norm{\tilde{\mathbf{U}}^{(n+1)}}^2_{H^s_x}.
\end{multline*}
Recall in particular the treatment of the term $\bar{u}\cdot \nabla_x\tilde{\mathbf{c}}^{(n+1)}$ which is done in $\eqref{a priori MS 2/3}$ and the use of Poincar\'e inequality $\eqref{poincare}$ which is now applied to $\tilde{\mathbf{c}}^{(n+1)}$ to recover the correct exponent for the Sobolev norm of $\tilde{\mathbf{U}}^{(n+1)}$ in the above estimate. Since $\norm{\bar{u}}_{H^s_x}\leq \delta_s$ and by construction also $\norm{\tilde{\mathbf{c}}^{(n)}}_{\sobolevx{s}} \leq \delta_s$, the definition of $\delta_s$ given in $\eqref{deltas}$ provides the same bounds that lead, through Gr\"onwall's lemma, to the final estimate
$$ \norm{\tilde{\mathbf{c}}^{(n+1)}}_{\sobolevx{s}} \leq  \norm{\tilde{\mathbf{c}}^\init}_{\sobolevx{s}} e^{-\lambda_s t} \leq \delta_s e^{-\lambda_s t}. $$
This proves that $\tilde{\mathbf{c}}^{(n+1)}$ belongs to $L^\infty\pab{0,T_0;H^s(\T^3)}$ and satisfies the iterative assumptions $\eqref{assumption iterative}$.

\medskip
By induction, we thus construct a sequence $\pa{\tilde{\mathbf{c}}^{(n)}}_{n\in\N}$ defined a.e. on ${(0,T_0)\times\T^3}$, belonging to $\pab{\mbox{Span}(\mathbf{1})}^\bot$, and bounded by $\delta_s$ in $L^\infty\pab{0,T_0;H^s(\T^3)}$. Moreover, the iterative equation $\eqref{iterative scheme}$ gives an explicit formula for $\partial_t \tilde{\mathbf{c}}^{(n+1)}$ in terms of $\tilde{\mathbf{c}}^{(n)}$, $\tilde{\mathbf{c}}^{(n+1)}$ and $\bar{u}$. Again, the continuous Sobolev embedding $H^{s}_x\hookrightarrow L^\infty_x$ for $s >3/2$ and Proposition \ref{prop:A-1} raise the existence of a polynomial $P$ in two variables, with coefficients only depending on $s$, $\mathbf{\bar{c}}$, $\norm{\bar{u}}_{H^{s}_x}$, $\lambda_A$ and $\mu_A$, such that
$$ \norm{\partial_t\tilde{\mathbf{c}}^{(n+1)}}_{L^2_x}\leq P\pa{\norm{\tilde{\mathbf{c}}^{(n)}}_{H^s_x},\norm{\tilde{\mathbf{c}}^{(n+1)}}_{H^s_x}} \leq P(\delta_s,\delta_s)\quad \textit{for a.e. } t \geq 0. $$
This shows that $\pa{\partial_t\tilde{\mathbf{c}}^{(n)}}_{n\in\N}$ is bounded in $L^\infty\pab{0,T_0;L^2(\T^3)}$, uniformly with respect to $n\in\N$.
\par Therefore, choosing $0<s'<s-2$, by Sobolev embeddings, there exists an $N$-vector function $\tilde{\mathbf{c}}^\infty\in L^\infty\pab{0,T_0;H^{s}(\T^3)}$ such that, up to a subsequence,
\begin{enumerate}
\item $\pa{\tilde{\mathbf{c}}^{(n)}}_{n\in\N}$ converges, weakly-* in $L^\infty(0,T_0)$ and weakly in $H^s_x$, to $\tilde{\mathbf{c}}^\infty$,\\[-0.5mm]
\item $\pa{\tilde{\mathbf{c}}^{(n)}}_{n\in\N}$, $\pa{\nabla_x\tilde{\mathbf{c}}^{(n)}}_{n\in\N}$ and $\pa{\nabla_x\nabla_x\tilde{\mathbf{c}}^{(n)}}_{n\in\N}$ converge weakly-* in $L^\infty(0,T_0)$ and strongly in $H^{s'}_x$, \\[-1mm]
\item $\pa{\partial_t \tilde{\mathbf{c}}^{(n)}}_{n\in\N}$ converges weakly-* in $L^\infty(0,T_0)$ and weakly in $L^2_x$.
\end{enumerate}
Integrating our scheme $\eqref{iterative scheme}$ against test functions, we can then pass to the limit as $n$ goes to $+\infty$ (the nonlinear terms being bounded and dealt with thanks to the strong convergences in $H^{s'}_x$), and we see that $\tilde{\mathbf{c}}^\infty$ is a weak solution to
\begin{equation*}
\begin{split}
\partial_t \tilde{\mathbf{c}}^\infty +& \mathbf{\bar{c}}\nabla_x\cdot\pa{A(\mathbf{c}^\infty)^{-1}\nabla_x\tilde{\mathbf{c}}^\infty - \frac{\langle \mathbf{c}^\infty,A(\mathbf{c}^\infty)^{-1} \nabla_x\tilde{\mathbf{c}}^\infty) \rangle}{\langle \mathbf{c}^\infty,\mathbf{1}\rangle}\mathbf{1}} + \bar{u}\cdot \nabla_x\tilde{\mathbf{c}}^\infty
\\[1mm]      &    \quad\quad\quad  + \eps \nabla_x\cdot\pa{\tilde{\mathbf{c}}^\infty\pa{A(\mathbf{c}^\infty)^{-1}\nabla_x\tilde{\mathbf{c}}^\infty - \frac{\langle \mathbf{c}^\infty,A(\mathbf{c}^\infty)^{-1} \nabla_x\tilde{\mathbf{c}}^\infty) \rangle}{\langle \mathbf{c}^\infty,\mathbf{1}\rangle}\mathbf{1}}} = 0.
\end{split}
\end{equation*}
Denoting $\tilde{\mathbf{U}}^\infty = A(\mathbf{c}^\infty)^{-1} \nabla_x\tilde{\mathbf{c}}^\infty$, this proves that $(\tilde{\mathbf{c}}^\infty,\tilde{\mathbf{U}}^\infty)$ is a weak solution to the system \eqref{perturbed MS mass}--\eqref{perturbed MS momentum}, belonging to $L^\infty\pab{0,T_0;H^s(\T^3)}\times L^\infty\pab{0,T_0;H^{s-1}(\T^3)}$. In particular, looking at equations \eqref{perturbed MS mass}--\eqref{perturbed MS momentum}, by means of the continuous embedding of $H^{s/2}_x$ in~$L^\infty_x$ one easily checks that $(\partial_t \tilde{\mathbf{c}}^\infty,\partial_t \tilde{\mathbf{U}}^\infty)$ belongs to $L^\infty\pab{0,T_0;L^2(\T^3)}\times L^\infty\pab{0,T_0;L^2(\T^3)}$ as soon as $s \geq 4$. Applying the Aubin-Lions-Simon theorem (see for example \cite[Theorem II.5.16]{BoyFab}), we thus also ensure that $(\tilde{\mathbf{c}}^\infty,\tilde{\mathbf{U}}^\infty)$ belongs to $C^0\pab{[0,T_0]; H^{s-1}(\T^3)}\times C^0\pab{[0,T_0]; H^{s-2}(\T^3)}$ for any $s>3$.

\smallskip
Therefore, using the continuity of $\tilde{\mathbf{c}}^\infty$, we can finally conclude thanks to the \textit{a priori} estimates established in Proposition \ref{prop:a priori MS}, which state that $\norm{\tilde{\mathbf{c}}^\infty(T_0)}_{H^s_x}\leq \delta_s$. Indeed, we can restart our scheme at $T_0$ from this initial condition and we can obtain a solution on the time interval $[T_0,2T_0]$. Again, using the continuity of $\tilde{\mathbf{c}}^\infty$ with respect to $t\in [T_0,2T_0]$ and Proposition \ref{prop:a priori MS}, the corresponding sequence will be bounded by $\delta_s$ at $2T_0$, and by induction we can construct a weak solution of \eqref{perturbed MS mass}--\eqref{perturbed MS momentum} on $[0,+\infty)$.
\end{proof}

\smallskip
In the next result we conclude by recovering the uniqueness of the solution to the orthogonal system $\eqref{perturbed MS mass}-\eqref{perturbed MS momentum}$. We remind the reader that this property has to be understood in a perturbative sense, since we are only able to prove the uniqueness of the fluctuations~${(\tilde{\mathbf{c}},\tilde{\mathbf{U}})}$ around the macroscopic equilibrium state $(\bar{\mathbf{c}},\mathbf{0})$. 
 
\smallskip
\begin{prop}\label{lem:uniqueness MS}
Let $s>3$ be an integer, and consider a couple $(\tilde{\mathbf{c}}^\init,\tilde{\mathbf{U}}^\init)$ satisfying the assumptions of Theorem \ref{theo:Cauchy MS orthogonal}. There exists $\delta_s >0$ such that, if $(\tilde{\mathbf{c}},\tilde{\mathbf{U}})$ and $(\tilde{\mathbf{d}},\tilde{\mathbf{W}})$ are two solutions of \eqref{perturbed MS mass}--\eqref{perturbed MS momentum} having the same initial datum $(\tilde{\mathbf{c}}^\init,\tilde{\mathbf{U}}^\init)$, then $\tilde{\mathbf{c}}=\tilde{\mathbf{d}}$ and $\tilde{\mathbf{U}}=\tilde{\mathbf{W}}$.
\end{prop}

\begin{proof}[Proof of Proposition \ref{lem:uniqueness MS}]
Subtracting the two sets of equations satisfied by $(\tilde{\mathbf{c}},\tilde{\mathbf{U}})$ and $(\tilde{\mathbf{d}},\tilde{\mathbf{W}})$, and denoting $\tilde{\mathbf{h}}=\tilde{\mathbf{c}}-\tilde{\mathbf{d}}$ and $\tilde{\mathbf{R}}=\tilde{\mathbf{U}}-\tilde{\mathbf{W}}$, we initially establish the relations
\begin{eqnarray}
&&\partial_t \tilde{\mathbf{h}} + \mathbf{\bar{c}}\nabla_x\cdot \mathbf{V}_{\tilde{R}} + \bar{u}\cdot \nabla_x\tilde{\mathbf{h}}+ \eps \nabla_x\cdot\pa{\tilde{\mathbf{h}}\VV} +\eps \nabla_x\cdot\pa{\tilde{\mathbf{d}}\mathbf{V}_{\tilde{R}}} = 0, \label{Uniqueness1}
\\[3mm]      &&  \nabla_x\tilde{\mathbf{h}} = A(\mathbf{c})\tilde{\mathbf{R}} + \cro{A(\mathbf{c})-A(\mathbf{d})}\tilde{\mathbf{W}}, \label{Uniqueness2}
\end{eqnarray}
with an obvious meaning for the shorthand $\mathbf{V}_{\tilde{R}}$.

We shall give similar computations to the ones derived for the \textit{a priori} estimates, except that we here restrict our investigation to the sole $L^2_x$ setting, since it will prove itself to be sufficient in order to deduce uniqueness. However, we still need the solutions to be in $H^s_x$ for some $s > 3$, in order to again take advantage of the Sobolev embedding $H^{s/2}_x\hookrightarrow L^\infty_x$. We compute the scalar product between $\mathbf{\bar{c}}^{-1}\tilde{\mathbf{h}}$ and the equation $\eqref{Uniqueness1}$, and we integrate over the torus. As in the proof of Proposition \ref{prop:a priori MS}, we use the gradient equation $\eqref{Uniqueness2}$ and its orthogonal properties to recover
\begin{eqnarray*}
\frac{1}{2}\frac{\dd}{\dd t}\norm{\tilde{\mathbf{h}}}^2_{L^2_x\big(\mathbf{\bar{c}}^{-\frac{1}{2}}\big)} &\leq& \int_{\T^3}\langle A(\mathbf{c})\tilde{\mathbf{R}},\tilde{\mathbf{R}} \rangle \dd x + \int_{\T^3} \langle \cro{A(\mathbf{c})-A(\mathbf{d})}\tilde{\mathbf{W}},\tilde{\mathbf{R}} \rangle_{\mathbf{\bar{c}}^{-1}} \dd x
\\[1mm]    &\:& + \int_{\T^3} \langle \nabla_x\tilde{\mathbf{h}},\tilde{\mathbf{h}}\pa{\bar{\mathbf{u}} + \eps\mathbf{V}_{\tilde{R}}}\rangle_{\mathbf{\bar{c}}^{-1}}\dd x +\eps\int_{\T^3} \langle \nabla_x\tilde{\mathbf{h}},\tilde{\mathbf{d}}\mathbf{V}_{\tilde{R}}\rangle_{\mathbf{\bar{c}}^{-1}}\dd x.
\end{eqnarray*}
We use the spectral gap of $A(\mathbf{c})$ for the first term on the right-hand side, while the remaining terms are dealt with thanks to the \textit{a priori} estimates derived in Proposition \ref{prop:a priori MS} and the usual Sobolev embedding, in the following way:
$$ \abs{\tilde{\mathbf{c}}(t,x)} \leq \norm{\tilde{\mathbf{c}}}_{H^s_x}\leq  \delta_s, \quad \abs{\bar{u}(t,x)} \leq \delta_s, \quad \big|\tilde{\mathbf{d}}(t,x) \big | \leq \delta_s, \quad \textit{on }\  \R_+\times\T^3. $$
This initially gives
\begin{equation}\label{uniqueness start}
\begin{split}
\frac{1}{2}\frac{\dd}{\dd t}\norm{\tilde{\mathbf{h}}}^2_{L^2_x\big(\mathbf{\bar{c}}^{-\frac{1}{2}}\big)}  \leq& -\frac{\lambda_A}{2}\norm{\tilde{\mathbf{R}}}^2_{L^2_x} + \frac{\sum\limits_{1\leq i,j \leq N} \disp{\int_{\T^3}} \abs{\tilde{R}_i} \abs{c_i c_j-d_i d_j}\abs{\tilde{W}_j-\tilde{W}_i} \:\dd x}{\min\limits_{1\leq i \leq N}\bar{c}_i \min\limits_{1\leq i,j\leq N}\Delta_{ij}}\hspace{2.8cm}
\\[1mm]     &   +\cro{\delta_s + \eps\delta_s\pa{1+\frac{N\delta_s^2}{C_0}}}\norm{\tilde{\mathbf{h}}}_{L^2_x}\norm{\nabla_x\tilde{\mathbf{h}}}_{L^2_x} + \eps\delta_s\norm{\nabla_x\tilde{\mathbf{h}}}_{L^2_x}\norm{\tilde{\mathbf{R}}}_{L^2_x}.
\end{split}
\end{equation}
Then, the algebraic manipulation
$$\abs{c_i c_j-d_i d_j} = \abs{\frac{1}{2}\pa{c_i-d_i}\pa{c_j+d_j} + \frac{1}{2}\pa{c_i+d_i}\pa{c_j-d_j}} \leq \eps\frac{\delta_s}{2}\pa{\abs{h_i} + \abs{h_j}}$$
and the Cauchy-Schwarz inequality yield the control
\begin{equation}\label{control sum uniqueness}
\sum\limits_{1\leq i,j \leq N} \int_{\T^3} \abs{\tilde{R}_i} \abs{c_i c_j-d_i d_j}\abs{\tilde{W}_j-\tilde{W}_i} \:\dd x \leq 2\eps \delta_s^2 \norm{\tilde{\mathbf{R}}}_{L^2_x}\norm{\tilde{\mathbf{h}}}_{L^2_x}.
\end{equation}
From the gradient relation $\eqref{Uniqueness2}$ and from the Poincar\'e inequality $\eqref{poincare}$, we also deduce the existence of a constant $C_s >0$ such that
\begin{equation}\label{control R h}
\norm{\tilde{\mathbf{R}}}_{L^2_x} \geq C_s\norm{\nabla_x\tilde{\mathbf{h}}}-\eps\delta_s^2\norm{\tilde{\mathbf{h}}}_{L^2_x} \geq \pa{C_s-\eps\delta_s^2}\norm{\tilde{\mathbf{h}}}_{L^2_x}.
\end{equation}
We now use $\eqref{control sum uniqueness}$, $\eqref{control R h}$ and the fact that $0< \eps \leq 1$ inside $\eqref{uniqueness start}$ to finally infer the upper bound
$$\frac{1}{2}\frac{\dd}{\dd t}\norm{\tilde{\mathbf{h}}}^2_{L^2_x\big(\mathbf{\bar{c}}^{-\frac{1}{2}}\big)} \leq \pa{-\frac{\lambda_A}{2}+\delta_s K(\delta_s)} \norm{\tilde{\mathbf{R}}}^2_{L^2_x}$$
where $K(\delta_s)>0$ is a polynomial in $\delta_s$ whose coefficients only depend on $\mathbf{\bar{c}}$ and on the number of species $N$. By choosing $\delta_s$ small enough so that both Proposition \ref{prop:a priori MS} holds and the inequality $-\frac{\lambda_A}{2}+\delta_s K(\delta_s) \leq 0$ is satisfied, we conclude that $\Vert \tilde{\mathbf{h}}\Vert_{L^2_x\big(\mathbf{\bar{c}}^{-\frac{1}{2}}\big)}$ decreases over time. Therefore, since initially $\tilde{\mathbf{h}}^\init=0$, we deduce that $\tilde{\mathbf{h}}=0$ at any time $t\geq 0$.
\par This implies that $\tilde{\mathbf{c}}=\tilde{\mathbf{d}}$, from which we also deduce that the gradient relation $\eqref{Uniqueness2}$ becomes
$$A(\mathbf{c})\tilde{\mathbf{R}} =0.$$
We thus infer that $\tilde{\mathbf{R}} = 0$, since $\tilde{\mathbf{R}}\in (\Span(\mathbf{1}))^\bot$. Consequently, $\tilde{\mathbf{U}} = \tilde{\mathbf{W}}$ and the uniqueness is established.
\end{proof}

\bigskip
\noindent \textbf{Step 5 -- Conclusion.} We are finally able to end our study of the incompressible Maxwell-Stefan system \eqref{MS mass vectorial}--\eqref{MS momentum vectorial}--\eqref{MS incompressibility vectorial}. Theorem \ref{theo:Cauchy MS orthogonal} is a direct gathering of Proposition \ref{prop:a priori MS}, Lemmata \ref{lem:preservations} and \ref{lem:positivity}, and Propositions \ref{lem:existence MS}--\ref{lem:uniqueness MS}.

\smallskip
Our main Theorem \ref{theo:Cauchy MS} then directly follows from Theorem \ref{theo:Cauchy MS orthogonal} with the unique orthogonal writing $\eqref{Condition2}$ established in Proposition \ref{prop:MS orthogonal writing}. In fact, as soon as the unique solution $(\bar{\mathbf{c}}+\eps\tilde{\mathbf{c}},\bar{\mathbf{u}}+\eps\tilde{\mathbf{U}})$ of the orthogonal system \eqref{perturbed MS mass}--\eqref{perturbed MS momentum} is established, the corresponding unique perturbative solution of the Maxwell-Stefan system \eqref{MS mass vectorial}--\eqref{MS momentum vectorial} with incompressibility condition $\eqref{MS incompressibility vectorial}$ is given by $(\bar{\mathbf{c}}+\eps\tilde{\mathbf{c}},\bar{\mathbf{u}}+\eps\tilde{\mathbf{u}})$, where
\begin{equation*}
\tilde{\mathbf{u}} = \tilde{\mathbf{U}} - \frac{1}{C_0}\langle \mathbf{c},\tilde{\mathbf{U}} \rangle\mathbf{1}
\end{equation*}
satisfies $\langle \mathbf{c},\tilde{\mathbf{u}} \rangle = 0$. In particular, the exponential decay of $\tilde{\mathbf{u}}$ directly follows from the exponential decays of $\tilde{\mathbf{c}}$ and $\tilde{\mathbf{U}}$.
\bigskip

\bibliographystyle{acm}
\bibliography{Bibliography_PerturbedIMS}

\begin{thebibliography}{10}

\bibitem{Ama}
{\sc Amann, H.}
\newblock {\em Nonhomogeneous linear and quasilinear elliptic and parabolic
  boundary value problems}.
\newblock In: Function spaces, differential operators and nonlinear analysis,
  pp. 9--126. H.J. Schmeisser, H. Triebel (eds), Teubner, Stuttgart, Leipzig,
  1993.

\bibitem{BonBri}
{\sc Bondesan, A., and Briant, M.}
\newblock Stability of the {M}axwell-{S}tefan system in the diffusion
  asymptotics of the {B}oltzmann multi-species equation.
\newblock {\em Comm. Math. Phys. 382}, 1 (2021), 381--440.

\bibitem{Bot}
{\sc Bothe, D.}
\newblock On the {M}axwell-{S}tefan approach to multicomponent diffusion.
\newblock In {\em Parabolic problems}, vol.~80 of {\em Progr. Nonlinear
  Differential Equations Appl.} Birkh\"{a}user/Springer Basel AG, Basel, 2011,
  pp.~81--93.

\bibitem{BouGotGre}
{\sc Boudin, L., G\"{o}tz, D., and Grec, B.}
\newblock Diffusion models of multicomponent mixtures in the lung.
\newblock In {\em C{EMRACS} 2009: {M}athematical modelling in medicine},
  vol.~30 of {\em ESAIM Proc.} EDP Sci., Les Ulis, 2010, pp.~90--103.

\bibitem{BouGrePav}
{\sc Boudin, L., Grec, B., and Pavan, V.}
\newblock The {M}axwell-{S}tefan diffusion limit for a kinetic model of
  mixtures with general cross sections.
\newblock {\em Nonlinear Anal. 159\/} (2017), 40--61.

\bibitem{BouGreSal1}
{\sc Boudin, L., Grec, B., and Salvarani, F.}
\newblock A mathematical and numerical analysis of the {M}axwell-{S}tefan
  diffusion equations.
\newblock {\em Discrete Contin. Dyn. Syst. Ser. B 17}, 5 (2012), 1427--1440.

\bibitem{BouGreSal2}
{\sc Boudin, L., Grec, B., and Salvarani, F.}
\newblock The {M}axwell-{S}tefan diffusion limit for a kinetic model of
  mixtures.
\newblock {\em Acta Applicandae Mathematicae 136}, 1 (2015), 79--90.

\bibitem{BoyFab}
{\sc Boyer, F., and Fabrie, P.}
\newblock {\em Mathematical tools for the study of the incompressible
  {N}avier-{S}tokes equations and related models}, vol.~183 of {\em Applied
  Mathematical Sciences}.
\newblock Springer, New York, 2013.

\bibitem{Cha}
{\sc Chang, H.}
\newblock Multicomponent diffusion in the lung.
\newblock {\em Fed. proc. 39}, 10 (1980), 2759--2764.

\bibitem{ChaCow}
{\sc Chapman, S., and Cowling, T.~G.}
\newblock {\em The {M}athematical {T}heory of {N}on-uniform {G}ases}.
\newblock Cambridge University Press, Cambridge, 1970.

\bibitem{CheJun}
{\sc Chen, X., and J\"{u}ngel, A.}
\newblock Analysis of an incompressible {N}avier-{S}tokes-{M}axwell-{S}tefan
  system.
\newblock {\em Comm. Math. Phys. 340}, 2 (2015), 471--497.

\bibitem{DauJunTan}
{\sc Daus, E., J\"{u}ngel, A., and Tang, B.~Q.}
\newblock Exponential time decay of solutions to reaction-cross-diffusion
  systems of {M}axwell-{S}tefan type.
\newblock {\em Archive Rat. Mech. Anal. 235\/} (2020), 1059--1104.

\bibitem{DesLepMou}
{\sc Desvillettes, L., Lepoutre, T., and Moussa, A.}
\newblock Entropy, duality, and cross diffusion.
\newblock {\em SIAM J. Math. Anal. 46}, 1 (2014), 820--853.

\bibitem{Fic}
{\sc Fick, A.}
\newblock Ueber {D}iffusion.
\newblock {\em Ann. der Physik 170\/} (1855), 59--86.

\bibitem{Gio2}
{\sc Giovangigli, V.}
\newblock Mass conservation and singular multicomponent diffusion algorithms.
\newblock {\em IMPACT Comput. Sci. Eng. 2\/} (1990), 73--97.

\bibitem{Gio1}
{\sc Giovangigli, V.}
\newblock {\em Multicomponent flow modeling}.
\newblock Modeling and Simulation in Science, Engineering and Technology.
  Birkh\"{a}user Boston, Inc., Boston, MA, 1999.

\bibitem{GioMas}
{\sc Giovangigli, V., and Massot, M.}
\newblock Les mélanges gazeux réactifs, ({I}) {S}ymétrisation et existence
  locale.
\newblock {\em C. R. Acad. Sci. Paris 323\/} (1996), 1153--1158.

\bibitem{HMPW17}
{\sc Herberg, M., Meyries, M., Pr\"{u}ss, J., and Wilke, M.}
\newblock Reaction-diffusion systems of {M}axwell-{S}tefan type with reversible
  mass-action kinetics.
\newblock {\em Nonlinear Anal. 159\/} (2017), 264--284.

\bibitem{HutSal1}
{\sc Hutridurga, H., and Salvarani, F.}
\newblock On the {M}axwell-{S}tefan diffusion limit for a mixture of monatomic
  gases.
\newblock {\em Math. Meth. in Appl. Sci. 40}, 3 (2017), 803--813.

\bibitem{HutSal2}
{\sc Hutridurga, H., and Salvarani, F.}
\newblock Existence and uniqueness analysis of a non-isothermal cross-diffusion
  system of {M}axwell-{S}tefan type.
\newblock {\em Appl. Math. Lett. 75\/} (2018), 108--113.

\bibitem{Jun}
{\sc J\"{u}ngel, A.}
\newblock The boundedness-by-entropy method for cross-diffusion systems.
\newblock {\em Nonlinearity 28}, 6 (2015), 1963--2001.

\bibitem{JunSte}
{\sc J\"{u}ngel, A., and Stelzer, I.~V.}
\newblock Existence analysis of {M}axwell-{S}tefan systems for multicomponent
  mixtures.
\newblock {\em SIAM J. Math. Anal. 45}, 4 (2013), 2421--2440.

\bibitem{Kawashima84}
{\sc Kawashima, S.}
\newblock Systems of a hyperbolic-parabolic composite type, with applications
  to the equations of magnetohydrodynamics.
\newblock {\em Doctoral Thesis, Kyoto Univ.\/} (1984).

\bibitem{KriWes}
{\sc Krishna, R., and Wesselingh, J.~A.}
\newblock The {M}axwell-{S}tefan approach to mass transfer.
\newblock {\em Chem. Eng. Sci. 52\/} (1997), 861--911.

\bibitem{LouMar}
{\sc Lou, Y., and Mart\'{\i}nez, S.}
\newblock Evolution of cross-diffusion and self-diffusion.
\newblock {\em J. Biol. Dyn. 3}, 4 (2009), 410--429.

\bibitem{LouNi}
{\sc Lou, Y., and Ni, W.-M.}
\newblock Diffusion, self-diffusion and cross-diffusion.
\newblock {\em J. Differential Equations 131}, 1 (1996), 79--131.

\bibitem{MajBer}
{\sc Majda, A.~J., and Bertozzi, A.~L.}
\newblock {\em Vorticity and incompressible flow}, vol.~27 of {\em Cambridge
  Texts in Applied Mathematics}.
\newblock Cambridge University Press, Cambridge, 2002.

\bibitem{MarTem}
{\sc Marion, M., and Temam, R.}
\newblock Global existence for fully nonlinear reaction-diffusion systems
  describing multicomponent reactive flows.
\newblock {\em J. Math. Pures Appl. 104\/} (2015), 102--138.

\bibitem{Max}
{\sc Maxwell, J.}
\newblock On the dynamical theory of gases.
\newblock {\em Phil. Trans. R. Soc. Lond. 157\/} (1867), 49--88.

\bibitem{ShiKawTer}
{\sc Shigesada, N., Kawasaki, K., and Teramoto, E.}
\newblock Spatial segregation of interacting species.
\newblock {\em J. Theoret. Biol. 79}, 1 (1979), 83--99.

\bibitem{Ste}
{\sc Stefan, J.}
\newblock {\"U}ber das gleichgewicht und die bewegung, insbesondere die
  diffusion von gasgemengen.
\newblock {\em Akad. Wiss. Wien 63\/} (1871), 63 -- 124.

\bibitem{TDBCH}
{\sc Thiriet, M., Douguet, D., Bonnet, J.-C., Canonne, C., and Hatzfeld, C.}
\newblock {The effect on gas mixing of a He-O2 mixture in chronic obstructive
  lung diseases}.
\newblock {\em Bull. Eur. Physiopathol. Respir. 15}, 5 (1979), 1053--1068.

\end{thebibliography}

\bigskip
\signandrea
\bigskip
\signmarc

\end{document}